\documentclass[a4wide]{article}

\usepackage{textcomp}
\usepackage[utf8]{inputenc}
\usepackage[T1]{fontenc}

\usepackage[%
    backend=biber,%
    doi=false,%
    url=false,%
    giveninits=true,%
    isbn=false,%
]{biblatex}
\bibliography{abssmooth-fw-arXiv_v2.bib}

\usepackage{hyperref}
\usepackage[dvipsnames]{xcolor}
\hypersetup{
    colorlinks,
    linkcolor={red!60!black},
    citecolor={green!40!black}
}

\usepackage[caption=false]{subfig}
\usepackage{wrapfig}

\usepackage{amssymb}
\usepackage{amsmath}
\usepackage{amsthm}
\usepackage{mathtools}
\usepackage{algorithm,algorithmic}
\usepackage{stmaryrd}
\usepackage{booktabs}
\usepackage{color}
\usepackage{nicefrac}
\usepackage{multirow}
\usepackage{lipsum}  
\usepackage{tikz}
\usepackage{pgfplots}
\usepackage{enumitem}

\definecolor{dblue}{HTML}{0455BF}
\definecolor{orng}{HTML}{D35400}
\definecolor{dred}{HTML}{D90404}
\definecolor{Dblue}{HTML}{8602DC}
\definecolor{bluep}{HTML}{0171BD}
\definecolor{pblue}{rgb}{0.1176,0.5647,1}
\definecolor{pgreen}{rgb}{0.1961,0.8039,0.1961}
\definecolor{pred}{rgb}{1.0,0.2706,0.0}
\definecolor{bred}{HTML}{FF0000}
\definecolor{bpurp}{HTML}{BF00BF}
\definecolor{bblu}{HTML}{0000FF}
\definecolor{bcyan}{HTML}{00BFBF}
\definecolor{byellow}{HTML}{BFBF00}
\definecolor{bgreen}{HTML}{008000}

\definecolor{lblue}{HTML}{0455BF}
\definecolor{dgreen}{HTML}{02724A}

\usepackage{todonotes}

\newif\ifTODO 
\TODOtrue

\definecolor{darkgreen}{rgb}{0,0.6,0}

\newtheorem{theorem}{Theorem}[section]
\newtheorem{corollary}[theorem]{Corollary}
\newtheorem{proposition}[theorem]{Proposition}
\newtheorem{lemma}[theorem]{Lemma}
\newtheorem{example}[theorem]{Example}

\newtheorem{definition}[theorem]{Definition}

\newtheorem{remark}[theorem]{Remark}

\newcommand{\diag}{\operatorname{diag}}

\newcommand{\cc}{\mathcal{C}}
\newcommand{\cO}{\mathcal{O}}
\newcommand{\rr}{\mathbb{R}}
\newcommand{\R}{\mathbb{R}}
\newcommand{\nn}{\mathbb{N}}
\newcommand{\NN}{\mathbb{N}}

\newcommand*{\vsepfbox}[1]{%
  \begingroup
    \sbox0{\fbox{#1}}%
    \setlength{\fboxrule}{0pt}%
    \mbox{\kern-\fboxsep\fbox{\unhbox0}\kern-\fboxsep}%
  \endgroup
}
\newcommand{\abs}{\operatorname{abs}}
\newcommand{\sign}{\operatorname{sign}}

\newcommand{\bx}{\bar{x}}
\DeclareMathOperator{\sgn}{sgn}

\DeclareMathOperator*{\argmin}{\mathrm{arg\,min}}

\title{On a Frank-Wolfe Approach for Abs-smooth Functions}

\author{%
    Timo Kreimeier,
    Sebastian Pokutta,\\
    Andrea Walther,
    and
    Zev Woodstock
}

\date{\today}

\begin{document}

\maketitle

\begin{abstract}
We propose an algorithm which appears to be the first bridge
between the fields of conditional gradient methods and abs-smooth
optimization. Our problem setting is motivated
by various applications that lead to nonsmoothness,
such as $\ell_1$ regularization, phase retrieval problems, or 
ReLU activation in machine learning. To handle the nonsmoothness
in our problem, we propose a generalization to the traditional
Frank-Wolfe gap and prove that first-order minimality is achieved
when it vanishes. We derive a convergence rate for our algorithm
which is {\em
identical} to the smooth case. Although our algorithm necessitates
the solution of a subproblem which is more challenging than the
smooth case, we provide an efficient numerical method for its
partial solution, and we identify several applications where our
approach fully solves the subproblem. Numerical and theoretical
convergence is demonstrated, yielding several conjectures.
\end{abstract}

\bigskip

\noindent
\section{Introduction}
\label{sec:intro}
Many applications, see, e.g., \textup{\cite{MBS21,TTP21}}, involve the
minimization of a function $f\colon\rr^n\to\rr$
subject to a compact and convex constraint $C\subset\rr^n$. That
is, one has to solve problems of the form
\begin{equation}
\label{e:p}
\underset{x\in C}{\text{min}}\;\; f(x)\;.
\end{equation}
We address \eqref{e:p} for the case when $f$ is an
{\em abs-smooth} function (see Definition~\ref{def:abssmooth}). In
a nutshell, the class of abs-smooth functions captures all
nonsmooth functions whose nondifferentiability arise as a result of
the absolute value function. Hence, this class also includes
smooth functions, $\max$, $\min$, and compositions/linear
combinations thereof. Since functions remain in this class when
used recursively \cite{Gr13}, one can readily show that many
important nonsmooth objective functions reside in this class like
piecewise linear models \cite{KuLiSt22}, $\ell_1$-regulariation as for LASSO problems
and phase retrival problems \cite{DaDrPa20}.
It was shown in \textup{\cite{Gr13}} that, for abs-smooth
functions, one can generate local piecewise linear models with
approximation properties 
up to second order similar in quality to a Taylor expansion; furthermore, all required
information to define these approximants are derivatives in the classical sense and
can be computed easily by
an extended version of algorithmic differentiation (AD)
\textup{\cite{GW08}}. By considering this subclass of nonsmooth
nonconvex objectives, efficient optimization algorithms with
guarantees such as first-order stationarity have been pioneered in
the last decade, see, e.g., \cite{FWG18,Gr13,GW18}. However,
it appears to be an open question as to whether or not one can
enforce closed convex constraints in an abs-smooth optimization
routine. In fact, even if one restricts to the subclass of {\em
piecewise linear} objective functions, it is currently only
possible to enforce piecewise linear constraints
\cite{KrWaGr21}.

To address this gap, we will draw upon the theory of Frank-Wolfe
(or {\em conditional gradient}) algorithms. In contrast to
more computationally-intensive methods which enforce the constraint
$C$ by evaluating proximity operators or projections, the
Frank-Wolfe algorithm, see, e.g.,
\textup{\cite{CGFWSurvey2022,FW56,CG66}}, only needs to
solve a linear optimization subproblem over $C$.
In traditional settings, the objective function $f$ is assumed to
be smooth, i.e., Lipschitz-continuously differentiable, and the
{\em Linear Minimization Oracle} (LMO) computes for the gradient
$c\equiv\nabla f(\bx)\in\rr^n$ at a current iterate $\bx\in \R^n$,
a point in $\argmin_{v\in C}\langle c\,,\, v\rangle$.
The Frank-Wolfe algorithm and its variants have gained popularity
because this linear minimization often requires fewer numerical
computations when compared to projection-based methods. For
instance, computing a projection onto
the spectrahedron $C=\{x\in\mathbb{S}^n_+\,|\,\text{Tr}(x)=1\}$
requires a full eigendecomposition; on the other hand, linear
minimization over $C$ only requires computing one dominant
eigenpair \textup{\cite{Ga16}}.

Frank-Wolfe approaches have been extensively studied for the smooth
case and various results are available for a myriad of settings
\textup{\cite{CBS21,CDP20,FW56,Ja13}}; see also
\textup{\cite{CGFWSurvey2022}} for an overview. However, despite
their significant computational advantages, to-date, conditional
gradient algorithms have rarely been studied outside of the smooth
setting. The approach proposed in \textup{\cite{O16}} extends
Frank-Wolfe methods to cover objective functions with
continuous, albeit non-Lipschitz, gradients.
However, for our applications, the target function $f$ is not
differentiable at all. A typical approach to overcome this problem
is to repeatedly minimize smoothed approximations of the objective
function \textup{\cite{Ne05}}. Then, the resulting algorithm
essentially encodes a proximity-type operator, which (as
demonstrated above) can be more costly than an LMO. Furthermore, as
a smoothed version of a nonsmooth function grows in fidelity to the
original, typically the Lipschitz constant of its gradient grows
arbitrarily large, so smooth optimization methods, whose rates
often depend on this smoothness constant, can exhibit poor behavior
in practice. Another approach to the nonsmooth setting relies on
the ability to compute a complete set of generalized derivatives of
$f$ at a given point \textup{\cite{RCH18}}, -- a capability which
is rarely available in practice.
The obvious approach, i.e., using a subgradient instead of a
gradient for the LMO model, has been shown to fail in general \textup{\cite[Example~1]{Ne18}}.
There are subgradient-based approaches \textup{\cite{AN22,GH16,O23}},
but they are restricted to the convex case or a somewhat special 
function class. It
appears that the analysis and theoretical tools from the abs-smooth
literature including algorithmic differentiation and piecewise linear
approximation have not been considered in conjunction with
conditional gradient algorithms. To our knowledge, this is 
the first work bridging this gap in the literature.

In this article, we propose a generalization of the Frank-Wolfe
algorithm for abs-smooth functions.
Our analysis broadens the traditional nonconvex
smooth setting of the Frank-Wolfe algorithm, which shows that an
optimality criterion known as the {\em Frank-Wolfe gap}, whose
definition relies on a gradient, asymptotically vanishes at a rate
of $\cO(1/\sqrt{t})$ \textup{\cite{P20}}. 
Due to the nonsmoothness inherent to our problem class, we propose
a generalization of the Frank-Wolfe gap for abs-smooth
functions that captures the original Frank-Wolfe gap as a special
case. We extend the current theory by proving that first-order
optimality is achieved when our generalized Frank-Wolfe gap vanishes.
Furthermore, we establish that our algorithm converges with a rate
which is {\em identical} to that of the Frank-Wolfe algorithm when
applied to nonconvex smooth objectives. This is consistent with
previous results, since the smooth Frank-Wolfe algorithm arises as
a particular case of our algorithm. 

While the smooth Frank-Wolfe setting requires the solution of a
linear minimization subproblem, it turns out that in general, our
nonsmooth analogue of the Frank-Wolf algorithm necessitates the
solution of a piecewise linear subproblem. In order to solve this
task, we adapt the Active Signature Method of \cite{GW18} to
produce a locally minimal solution of our subproblem. We also
establish that for several applications, including all constrained
LASSO problems, our subroutine Algorithm~\ref{alg:aasm} yields a
global solution to the piecewise linear subproblem. We show that
our new subroutine is computationally faster than other
state-of-the-art methods on constrained piecewise linear problems,
and we also benchmark our full algorithm on several standard
nonsmooth test problems. 

It is important to note that the ingredients of our method can be
easily provided once the function to be optimized is given as
computer program. Hence, it is readily applicable to such functions even
if the requirements of the convergence theory provided in this paper
can not be verified apriori. This is in contrast to many algorithms
for solving nonsmooth optimization problems. 

This article is structured as follows. In the remaining part of
this section, we introduce abs-smooth functions and their
properties.
In Section~\ref{sec:algo}, we present our Frank-Wolfe algorithm for
abs-smooth functions (Section~\ref{subsec:motivation}), propose the
generalized Frank-Wolfe gap, derive
guarantees of first-order optimality (Section~\ref{sec:opt}), and
provide convergence
guarantees for our algorithm (Section~\ref{sec:conv}). We also
discuss the wide
applicability of our approach to yield abs-smooth versions of many
variants of the vanilla Frank-Wolfe algorithm in
Remark~\ref{rmk:mods}. Section~\ref{sec:subprob} is dedicated to
the analysis and solution of our algorithm's piecewise linear
subproblem. Our strategy for solving the
piecewise linear subproblem is discussed in
Section~\ref{sec:subprobsolve}, and potential relaxations are
discussed in Section~\ref{sec:altsubprob}. Finally, numerical
results are shown in Section~\ref{sec:num}. A summary and
outlook are contained in Section~\ref{sec:final}. 

\subsection{Abs-smooth functions}
Throughout we will consider the following class of target
functions. 
\begin{definition}[$\cc_{\abs}^d(\rr^n)$, abs-smooth functions]
\label{def:abssmooth}
For any $d \in \nn$, the set of locally Lipschitz continuous
functions $f: \rr^n \to \rr, ~ y =f(x)$, defined by an
\emph{abs-smooth form}
\begin{align}
  \begin{split}
z & = F(x,z,|z|) \; , \\
y & = \varphi(x,z) \; ,
\end{split} \label{eq:abs}
\end{align}
with $F \in \cc^d(\rr^{n+s+s}, \rr^s)$ and $\varphi \in \cc^d(\rr^{n+s},\rr)$, such that $z_i$ is determined only by the values of $z_j$, $1 \leq j < i$, is denoted by $\cc^d_{\abs}(\rr^n)$. For any $d\ge 1$, a function $f\in \cc^d_{\abs}(\rr^n)$ is called \emph{abs-smooth}. The components $z_i$, $1 \leq i \leq s$, of $z$ are called
\emph{switching variables}.
\end{definition}
For $d=1$, Definition~\ref{def:abssmooth} encapsulates the original
definition in \textup{\cite{Gr13}} defining abs-smooth functions in a more formal way. This enables for example the convergence analysis contained in \cite{GW18}. For the results derived in this paper, any value of $d\ge 1$ suffice.
Despite the fact that the definition of abs-smoothness seems to be rather technical, the class of abs-smooth functions is quite broad. It encompasses a large subset of piecewise smooth functions in the sense of Scholtes
\textup{\cite{Sch12}} since the evaluation of $\max(.,.)$ and $\min(.,.)$ can be expressed
using the absolute value function.
 Furthermore, abs-smooth functions capture a wide variety of
nonsmooth
functions used in machine learning applications, e.g., the $\ell_1$-norm, i.e., 
\begin{align*}
 f:\R^n \mapsto \R,\quad  z_i = x_i, \;\; 1\le i\le n\;,\;\;z_{n+1} = \sum_{i=1}^n|z_i| \quad\mbox{and}\quad  f(x) =  z_{n+1},
\end{align*}
as well as the ReLU
activation function given by
\begin{align*}
  f:\R \mapsto \R,\quad z_1 = x, \; z_2 = |z_1|\quad\mbox{and}\quad  f(x) = \max(x,0)=0.5(x+|x|) = 0.5(x+z_2) = 0.5(z_1+z_2)\;.
\end{align*}
Using the recursive evaluation procedure from \textup{\cite{Gr13}},
it follows that the composition of abs-smooth functions remains
abs-smooth.
Therefore, using the absolute value function and smooth univariate
functions as building blocks,
one can combine them via recursion and linear combination to
construct a rich class of abs-smooth functions comprinsing, e.g., also 
$\operatorname{min}$ and $\operatorname{max}$. As a result, provided a neural network uses abs-smooth
activation functions, its resulting squared $\ell_2$-loss is also
abs-smooth. 

Also complementarity conditions or equilibrium constraints can be
formulated in an abs-smooth form via
\begin{align*}
  0 \le x \perp y \ge 0 \quad\Longleftrightarrow \quad 0 = f(x,y) = \min(x,y)=0.5(x+y+|x-y|), \; z_1 = x-y, \; z_2 = |z_1| \;.
\end{align*}
 Finally, it is important to note that, for the application of the generalized Frank-Wolfe method proposed in this paper, the user does not have to state the function evaluation in the form \eqref{eq:abs}, since correspondingly adapted AD tools can generate this representation from a straight-line code using only smooth operations and the absolute value in a completely automated fashion.

The main advantage of formulating all these applications as abs-smooth functions is the
localization of the nonsmoothness as argument of an otherwise smooth function.
In this way, the nonsmoothness can be explicitly exploited in combination with
standard smooth optimization theory. For example, if an abs-smooth function
is nonsmooth at a given point $x$ then at least one of the
switching variables as argument of the absolute value is evaluated
at zero -- motivating also the name for these variables.
\begin{example}[Simple example]
  The function $f:\R \mapsto\R, f(x) = \max(0,x,2x+1)$ is abs-smooth since it can be stated
  in the following form
  \begin{align*}
    f(x) = \max(0,x,2x+1) = 0.25(3x+1+|x+1|+|3x+1+|x+1||)
  \end{align*}
  such that one obtains as one abs-smooth representation
\begin{align*}
  z_1 & = x+1 \; , \\
z_2 & = 3x+1+|z_1| \; , \\
z_3 & = |z_1| + |z_2| \; ,\\
f(x) & = 0.25(3x+1+z_3) \; .
\end{align*}
As can be seen, at the only nonsmooth point $x=-0.5$, the switching variable $z_2$ is zero. Note, that it can happen that a switching variable is evaluated at zero but the function itself is smooth.   
\end{example}
\begin{example}[Mifflin II] The example Mifflin II given by the function
\label{ex:M2a}
  \begin{align}
    \label{eq:mif}
f : \R^2 \mapsto \R,\qquad f(x) = - x_1 +2\left(x_1^2+x_2^2-1\right)+1.75\left|x_1^2+x_2^2-1\right|\;,
  \end{align}
  see, e.g., \textup{\cite{BaKaMae14}}, is a well-established test case for nonsmooth optimization.
  It is abs-smooth since it has the representation
  \begin{align*}
    z_1 & = x_1^2+x_2^2-1\\
    z_2 & = |z_1|\\
    y   & = -x_1+2z_1+1.75z_2\;.
  \end{align*}  
\end{example}

\subsection{The abs-linearization}
For an abs-smooth function $f$ and a point $\bx$, Griewank
proposed in \textup{\cite{Gr13}} the so-called \emph{abs-linearization}
$\Delta f(\bx;\cdot)$ which can be used to construct a
{\em piecewise linear model}
$$f_{PL,\bx}(\cdot) \equiv f(\bx)+\Delta f(\bx; \cdot-\bx)\approx f(\bx+\cdot)\;.$$
This model can be viewed as a generalized Taylor
expansion at $\bx$ which simultaneously accounts for non\-smoothness 
and maintains second-order accuracy: 

\begin{theorem}
\label{theo:approx}
Suppose $f$ is abs-smooth on ${\cal D} \subset {\cal K} \subset \rr^n$, ${\cal D}$ open, ${\cal K}$ compact and convex. Then there exists $\gamma > 0$ such that for all $x,
\bx \in {\cal D}$
\begin{align}
\label{e:smooth}
\|f(x)- f_{PL,\bx}(x)\| = \|f(x) -f(\bx)- \Delta f(\bx; x-\bx)\| \leq \gamma \|x-\bx\|^2\;.
\end{align}
\end{theorem}
\begin{proof}
See \textup{\cite[Proposition~1]{Gr13}}.
\end{proof}
For detailed explanation of the methodologies for generating
$\Delta f$, we refer to \textup{\cite{Gr13}} as well as the
Algorithmic Differentiation tools like ADOL-C~\cite{adolc},
CppAD~\cite{cppad}, and Tapenade~\cite{tapenade} which have been
extended to generate abs-linearizations and local piecewise linear
models in an automated fashion. It is important to emphasize that,
once a function $f$ is available as a C or Fortran program, its
piecewise linear approximant $f_{PL;\bx}(\cdot)$ at a
given point $\bx$ is accessible in an easy way. 
For our proof-of-concept implementation, a dense-matrix representation
of the piecewise linear approximant $f_{PL;\bx}(\cdot)$ is used
requiring $\mathcal{O}((n+s)(s+1))$ memory. Hence, it is only feasible
for small to medium size problems. However, the involved derivative
matrices are usually very sparse such that an efficent implementation would
require $\mathcal{O}((s+1)N)$ space, where
$N$ is the maximum number of nonzero entries in a row and expected to be rather small.
We observed this behavior, e.g., for the robust optimization of the Greek gas network in \cite{KrLiSrWa22}.
Once the representation of $\Delta
f(\overline{x};\cdot)$ is stored for a base point $\overline{x}$,
the cost of re-evaluating the abs-linear form is quite cheap requiring only matrix-vector products.

In contrast to many optimization approaches for the nonsmooth setting,
our advantage is that the piecewise linearization approximates $f$
by explicitly taking its nonsmoothness into account.
For intuition of how the abs-linearization behaves geometrically,
we provide the abs-linearization of Example~\ref{ex:M2a}.

\begin{example}[Mifflin II, continued] 
\label{ex:M2b}
For the Mifflin II function, its abs-linearization is given by $ \Delta f(\bx;\cdot) : \R^2 \mapsto \R$ with
  \begin{align}
    \label{eq:mif_pl}
    \Delta f(\bx; \Delta x) \!  =\! -\! \Delta x_1\! +\!2\left(2 \bx_1 \Delta x_1\!+\!2 \bx_2\Delta x_2\right)\!+\!1.75\left(\left|\bx_1^2\!+\!\bx_2^2\!-\!1\!+\!2 \bx_1 \Delta x_1\!+\!2 \bx_2\Delta x_2\right|\!-\! \left|\bx_1^2\!+\!\bx_2^2\!-\!1\right|\right)\;.
  \end{align}
The function itself together with its piecewise linear
model $f_{PL,\bx}(.)$  at the point $\bx = (-1.8,1.8)$ are illustrated in
Figure~\ref{fig:mif}.
\end{example}
\begin{figure}[ht]
\begin{center}             
\includegraphics[width=6cm]{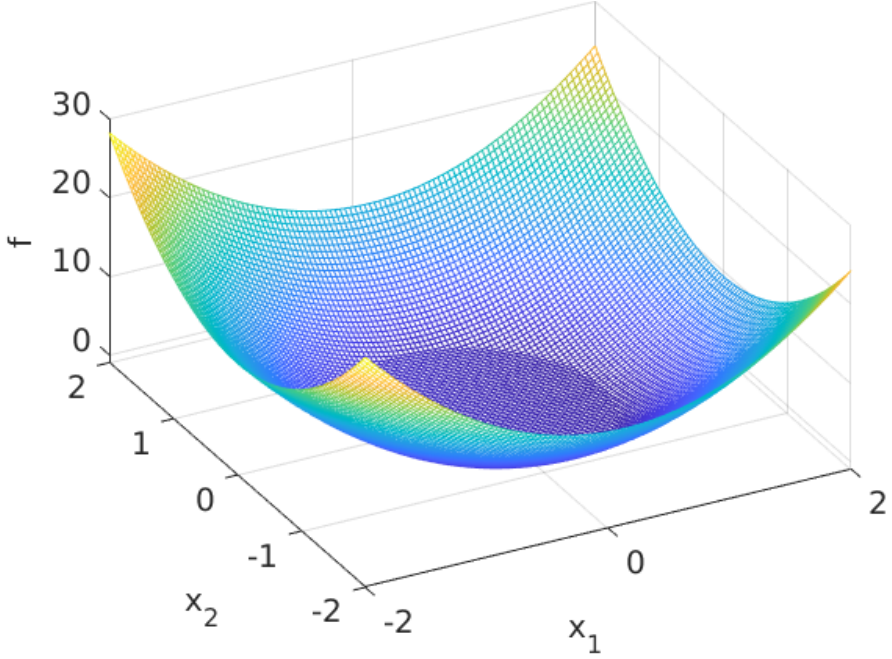}\qquad\qquad
\includegraphics[width=6cm]{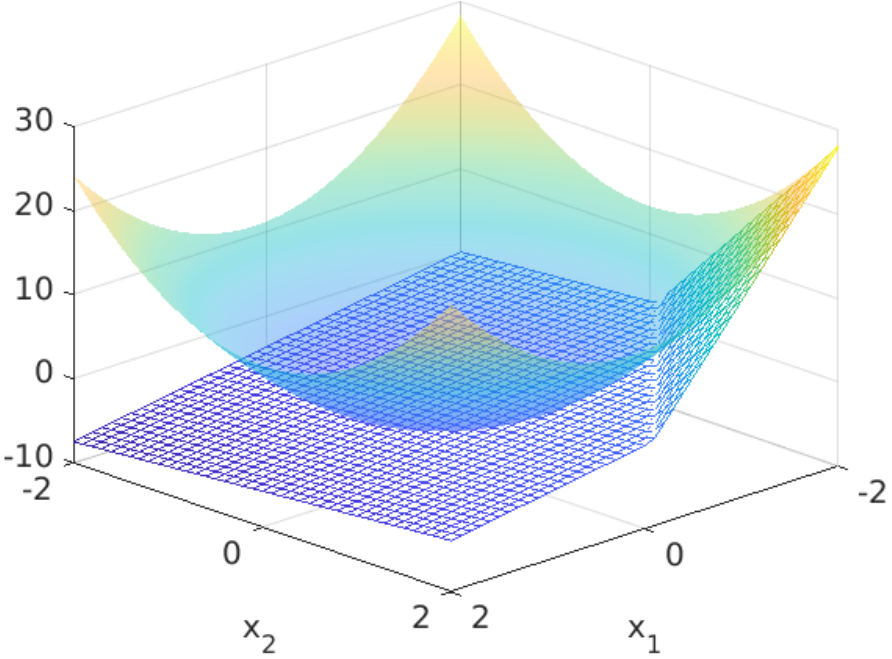}
\end{center}
\caption{Abs-smooth function \eqref{eq:mif} from
Example~\ref{ex:M2b} and its piecewise linear model
\eqref{eq:mif_pl}.}
\label{fig:mif}
\end{figure} 

While abs-smooth functions may lack gradients, they are guaranteed
to possess directional derivatives; furthermore, within a
neighborhood around the development point $\bx$, their directional
derivatives coincide with the 
abs-linearization:
\begin{proposition}
\label{prop:dir}
Let $f\in\cc_{\abs}^d(\rr^n)$ and $\bx\in\rr^n$. Then there
exists a constant $\rho>0$ such that, for every $d\in\rr^n$,
the directional (Bouligand) derivative $f'(\bx;d)$ exists and if
$\|d\|\leq\rho$ then
\begin{equation}
\label{e:dir}
\Delta f(\bx; d)= f'(\bx;d)\;.
\end{equation}
As a result, for every $\alpha\in[0,1]$, $\|d\|\leq\rho$ implies
positive homogeneity $\Delta f(\bx;\alpha d)=\alpha\Delta
f(\bx;d)$.
\end{proposition}
\begin{proof}
See \textup{\cite[Section~3.2]{Gr13}}, the final
claim follows from positive homogeneity of $f'(\bx;\cdot)$.
\end{proof}
Whether or not $\bx$ resides on a kink, $\rho$ basically describes
the distance to the next-closest kink of $\Delta f(\bx;\cdot)$. Hence, even though $\rho>0$ is guaranteed,
$\rho$ can tend towards $0$ as $\overline{x}$ approaches a point
where $f$ is nondifferentiable. Proposition~\ref{prop:dir}
illuminates the connection between the abs-linearization and a {\em
first-order minimal} point $\bx$ 
which satisfies, for all $v\in C$,
\begin{equation}
\label{e:fom}
f'(\bx;v-\bx)\geq 0\;.
\end{equation}

\section{A Frank-Wolfe algorithm for abs-smooth functions}
\label{sec:algo}

We begin in Section~\ref{subsec:motivation} by presenting our
Abs-Smooth Frank-Wolfe algorithm (ASFW) and describing the
guiding principles in its design. Sections~\ref{sec:opt} and
\ref{sec:conv} establish its theoretical footing by, respectively,
providing optimality conditions and convergence guarantees. When
the objective function $f$ is smooth, the original Frank-Wolfe
algorithm can be viewed as a special case of our algorithm; in this
setting, the optimality and convergence guarantees for the
nonconvex Frank-Wolfe algorithm would arise as corollaries to
our results.

\subsection{Motivating the algorithm}
\label{subsec:motivation}

A vanilla Frank-Wolfe algorithm for the smooth setting is
shown in Algorithm~\ref{alg:fw} (see also
\textup{\cite{CGFWSurvey2022}}).
The hallmark of the Frank-Wolfe algorithm is in Line~\ref{fw:2} --
the so-called LMO step which, at iteration $t\in\nn$, solves
\begin{equation}
\label{e:fwgapmax}
\underset{v\in C}{\text{max}}\;\;\langle -\nabla
f(x_t)\,,\,v-x_t\rangle\;.
\end{equation}
The optimal value of \eqref{e:fwgapmax} is called the {\em
Frank-Wolfe gap}. An advantage of Frank-Wolfe algorithms is that
the subproblem \eqref{e:fwgapmax} is oftentimes computationally
cheaper than evaluating the projection onto $C$. However, since
gradients are not available in the nonsmooth
setting,
this step of Frank-Wolfe algorithms must be changed. As far as we
are aware, there is no generic replacement of \eqref{e:fwgapmax} for
the fully nonsmooth setting of Frank-Wolfe algorithms. For instance,
Nesterov pointed out in \textup{\cite[Example~1]{Ne18}}  that the simple
approach of
replacing the gradient $\nabla f(x_t)$ in \eqref{e:fwgapmax} with a
subgradient -- similar to subgradient descent methods -- fails in
general. However, it was recently shown that a subgradient-based
approach does work for some nonsmooth functions~\textup{\cite{O23}}.

A key component in deriving a comprehensive convergence theory for
Algorithm~\ref{alg:fw} in the smooth nonconvex setting is the
{\em smoothness inequality}, which provides a quadratic
upper-bound on the difference between a function and its
first-order Taylor approximation. 
Even though we do not have this inequality, the abs-linearization
provides a piecewise linear model which also has second order
accuracy (Theorem~\ref{theo:approx}); in fact, this model {\em is}
the Taylor series approximation in the smooth case. Hence, it is
reasonable to consider a generalization of Algorithm~\ref{alg:fw},
Step~\ref{fw:2} which relies on the abs-linearization.

\begin{algorithm}[t]
\caption{Frank-Wolfe algorithm}
\label{alg:fw}  
\begin{algorithmic}[1]
\REQUIRE Point $x_0\in C$, smooth function $f$
\FOR{$t=0$ \textbf{to} $\dotsc$}
\STATE Choose step size $\alpha_t \in (0,1]$
\STATE Compute $v_t\in \argmin_{v\in C}\langle\nabla
f(x_t)\,,\,v\rangle$
\label{fw:2}
\STATE $x_{t+1}= (1-\alpha_t)x_t+\alpha_tv_t$ 
\ENDFOR
\end{algorithmic}
\end{algorithm}

Our approach in Algorithm~\ref{alg:asfw} is to (a) generalize the
Frank-Wolfe gap for abs-smooth functions, and (b) solve its
corresponding piecewise linear optimization problem as a subroutine
in our algorithm. As we further illustrate in
Section~\ref{sec:opt}, we propose the following {\em generalized Frank-Wolfe gap} for
abs-smooth functions
\begin{equation}
  \label{e:gengap}
  \max_{v\in C}\frac{-\Delta f(x_t;\alpha_t(v-x_t))}{\alpha_t}\;,
  \end{equation}
whose corresponding optimization problem is formally equivalent to
Line~\ref{a:s2} in Algorithm~\ref{alg:asfw}.

As we will see in Section~\ref{sec:opt}, just as in the smooth
setting, first-order minimality is achieved when the generalized Frank-Wolfe
gap vanishes. Furthermore, Section~\ref{sec:conv} demonstrates that
the approximation properties in Theorem~\ref{theo:approx} can be
leveraged to acquire identical convergence rates to the smooth
setting.

\begin{algorithm}[ht]
\caption{Abs-Smooth Frank-Wolfe (ASFW) algorithm}
\label{alg:asfw}  
\begin{algorithmic}[1]
\REQUIRE Point $x_0\in C$, abs-smooth function $f$
\FOR{$t=0$ \textbf{to} $\dotsc$}
\STATE Choose step size $\alpha_t \in (0,1]$
\STATE Compute $v_t\in\argmin_{v\in C}\Delta f(x_t,\alpha_t(v-x_t))$
\label{a:s2}
\STATE $x_{t+1}= (1-\alpha_t)x_t+\alpha_tv_t$ 
\ENDFOR
\end{algorithmic}
\end{algorithm}

\subsection{Optimality criterion: Generalizing the Frank-Wolfe gap}
\label{sec:opt}
In this section, we establish that, just as in 
the smooth setting, first-order optimality is acquired when
the generalized Frank-Wolfe gap \eqref{e:gengap} vanishes.

We note that \eqref{e:gengap} is consistent with the original
Frank-Wolfe gap \eqref{e:fwgapmax}. Indeed, if $f$ is smooth, then for
every $\bx\in\mathbb{R}^n$,
$\Delta f(\bx;\cdot)=\langle\nabla f(\bx)\,,\,\cdot\rangle$
\cite{Gr13}. Hence, under the assumption of smoothness, at
every iteration $t\in\nn$ we have
\begin{equation}
\label{e:8}
\dfrac{-\Delta f(x_t;\alpha_t(v_t-x_t))}{\alpha_t}=
\frac{-\langle\nabla f(x_t)\,,\,\alpha_t(v_t-x_t)\rangle}{\alpha_t}=
\langle-\nabla f(x_t)\,,\,v_t-x_t\rangle\;.
\end{equation}
In other words, our generalization captures the traditional
Frank-Wolfe gap \eqref{e:fwgapmax} as a special case.
This generalization from a linear gap to a nonlinear gap also
mirrors the form of the perspective function from nonlinear
analysis \textup{\cite{C18}}.

In the smooth setting, if the Frank-Wolfe gap vanishes this implies
first-order optimality. Hence, our goal is to show the same for our
generalization. To-date, it was only possible to
infer that $f$ is Clarke-stationary at $\bx$, if $\Delta
f(\bx;\cdot)$ was Clarke-stationary at $0$
\textup{\cite[Lemma~1]{FWG18}}.
However, since there are no guarantees that the solution to our
problem will lie on a vertex of $C$, we do not have a guarantee
that $v_t-x_t\to 0$. In fact, there are concrete examples where
this never occurs -- even for the smooth setting of
Algorithm~\ref{alg:fw} \cite[Figure~2.4]{CGFWSurvey2022}. We
therefore provide a new result below.

\begin{theorem}[First-order minimality]
\label{theo:fom1}
Let $C\subset\rr^n$ be nonempty and convex, let
$\bx\in C$, let $\alpha\in(0,1]$, and let
$f\in\cc_{\abs}^d(\rr^n)$. Suppose that 
\begin{equation}
\label{e:fom2}
\max_{v\in C}\dfrac{-\Delta f(\bx;\alpha(v-\bx))}{\alpha}=0\;.
\end{equation}
Then $f$ is first-order minimal at $\bx$, i.e., for every $v\in C$,
$f'(\bx;v-\bx)\geq 0$.
In particular, $\min_{v\in C}\Delta
f(\bx;v-\bx)=0$ implies first-order minimality at $\bx$.
\end{theorem}
\begin{proof}
For the sake of contradiction, suppose that there exists a point
$v\in C$ yielding a descent direction $f'(\bx;v-\bx)<0$.
Recall that Proposition~\ref{prop:dir} guarantees $\Delta
f(\bx;\cdot)=f'(\bx;\cdot)$ for arguments with norm bounded by
$\rho>0$. So, for $\tau\in(0,1)$ satisfying
$\alpha\tau\|v-\bx\|\leq\rho$ we have $\alpha\tau
v+(1-\alpha\tau)\bx\in C$ and hence
\begin{align*}
\dfrac{-\Delta f(\bx;(\alpha\tau v + (1-\alpha\tau)\bx) -
\bx)}{\alpha} =
\dfrac{-\Delta f(\bx;\alpha\tau (v-\bx))}{\alpha}
=\dfrac{-f'(\bx;\alpha\tau (v-\bx))}{\alpha}
=-\tau f'(\bx;v-\bx)
>0\;,
\end{align*}
which is absurd since it contradicts \eqref{e:fom2}.
\end{proof}
Note that first-order minimality is sometimes also called
d-stationarity, see, e.g., \textup{\cite{O23}}. Hence, if
\eqref{e:fom2} holds for $\bx$, then $\bx$ is d-stationary.
\begin{corollary}[Convex optimality]
Let $C\subset\rr^n$ be a nonempty compact convex set, let $\bx\in
C$, and let $f\in\cc_{\abs}^d(\rr^n)$ be convex. Suppose that
$\min_{v\in C}\Delta f(\bx;v-\bx)=0$. Then $f(\bx)=\min_{x\in C}
f(x)$.
\end{corollary}
\begin{proof}
  From Theorem~\ref{theo:fom1}, we can conclude that
  $$f'(\bx;v-\bx)\geq 0\qquad \forall v\in C\;,$$
which is a sufficient optimality condition for convex functions.
\end{proof}

\subsection{Convergence results}
\label{sec:conv}

In this section, we provide convergence proofs for
Algorithm~\ref{alg:asfw} under various settings. As we will see in
this
section, our $\cO(1/\sqrt{t})$ convergence results achieve the same
optimal rate as for the nonconvex {\em smooth} setting
\textup{\cite{P20}},
even though our functions are nonsmooth. For this purpose, we show
some important basic properties.
\begin{corollary}[Sign of $\Delta f(x;\alpha(v_*-x))$]
\label{cor:sign}
Let $x\in C$, let $\alpha>0$, and let $f\in\cc_{\abs}^d(\rr^n)$.
Suppose that $v_*\in C$ satisfies one of the following statements
\begin{enumerate}
\item
\label{cor:signi}
$\Delta f(x,\alpha(v_*-x)) \le\Delta f(x,0)$
\item
\label{cor:signii}
$v_*\in\argmin_{v\in C}\Delta f(x;\alpha(v-x))$.
\end{enumerate}
Then $\Delta f(x;\alpha(v_*-x))\leq 0$.
\end{corollary}
\begin{proof}
Theorem~\ref{theo:approx} implies that $\Delta f(x;0)=0$ proving
\ref{cor:signi}. To show \ref{cor:signii}, since $x\in C$, we know
\begin{align*}
\Delta f(x;\alpha(v_*-x)) \le \Delta f(x;\alpha(x-x)) = \Delta
f(x;0) \; ,
\end{align*}
so using \ref{cor:signi} we obtain the required estimate.
\end{proof}

\begin{lemma}
\label{l:t}
Let $C$ be a nonempty compact convex set with diameter
$D\in\rr_{\geq 0}$. Assume that $f\in\cc_{\abs}^d(\rr^n)$. Then,
for every $t\in\mathbb{N}$, the iterates generated by
Algorithm~\ref{alg:asfw} satisfy
\begin{align}
\label{e:smooth2}
0&\leq
\frac{-\Delta f(x_t;\alpha_t(v_t-x_t))}{\alpha_t}
\leq\frac{f(x_t)-f(x_{t+1})}{\alpha_t}+\alpha_t\gamma D^2
\\
\label{e:smoothsum}
t\min_{\substack{0\leq k\leq {t-1}}}
\frac{-\Delta f(x_k;\alpha_k(v_k-x_k))}{\alpha_k}&\leq
\sum_{k=0}^{t-1}\frac{-\Delta f(x_k;\alpha_k(v_k-x_k))}{\alpha_k}
\leq
\sum_{k=0}^{t-1}\frac{f(x_k)-f(x_{k+1})}{\alpha_k}
+\gamma D^2\sum_{k=0}^{t-1}\alpha_k\;.
\end{align}
\end{lemma}
\begin{proof}
  Since $C$ is compact, one can find an open set  ${\cal D}$ with $C\subset  {\cal D}$. Then,
  it follows from Theorem~\ref{theo:approx} that there exists a $\gamma >0$ such that
  \begin{align*}
    -\Delta f(x_t;\alpha_t(v_t-x_t))\leq
f(x_t)-f(x_{t+1})+\gamma\|x_{t+1}-x_t\|^2
  \end{align*}
  independent of the iteration counter $t$. 
Now, Corollary~\ref{cor:sign} yields
\begin{align*}
0\leq-\Delta f(x_t;\alpha_t(v_t-x_t))\leq
f(x_t)-f(x_{t+1})+\gamma\|x_{t+1}-x_t\|^2=
f(x_t)-f(x_{t+1})+\alpha_t^2\gamma\|v_t-x_t\|^2\;.
\end{align*}
Division by $\alpha_t$ yields \eqref{e:smooth2}, and summing over
all iterations from $0$ to $t-1$ yields \eqref{e:smoothsum}.
\end{proof}

These properties already suffice to show the first nonconvex
convergence result, where no additional assumptions on the
piecewise linear model are required.

\begin{theorem}[Open-loop convergence]
\label{theo:t}
Let $C$ be a nonempty compact convex set with diameter
$D\in\rr_{\geq 0}$. Assume that $f\in\cc_{\abs}^d(\rr^n)$.
Then, for every $t\in\mathbb{N}$, the iterates generated by Algorithm~\ref{alg:asfw} with
$\alpha_t=1/\sqrt{1+t}$ satisfy
\begin{align}
\label{e:open-ncv}
0\leq
\min_{\substack{0\leq k\leq {t-1}}}
\frac{-\Delta f\left(x_k;\alpha_k(v_k-x_k)\right)}{\alpha_k}\leq
\frac{1}{t}\sum_{k=0}^{t-1}\dfrac{-\Delta
f(x_k;\alpha_k(v_k-x_k))}{\alpha_k}
\leq\mathcal{O}\left(\frac{1}{\sqrt{t}}\right)\;.
\end{align}
\end{theorem}
\begin{proof}
Let $x_*$ be a minimizer of $f$ over $C$, and for every $t\in\NN$,
let $g_t=f(x_t)-f(x_*)$. Corollary~\ref{cor:sign} and
Lemma~\ref{l:t} yield
\begin{equation}
\label{e:ol1}
  \begin{split}
0 &\leq \min_{\substack{0\leq k\leq {t-1}}}
\frac{-\Delta f(x_k;\alpha_k(v_k-x_k))}{\alpha_k}
\leq
\frac{1}{t}\sum_{k=0}^{t-1}\frac{-\Delta
f(x_k;\alpha_k(v_k-x_k))}{\alpha_k}\\
& \leq
\frac{1}{t}\left(\sum_{k=0}^{t-1}\frac{g_k-g_{k+1}}{\alpha_k}
+\gamma D^2\sum_{k=0}^{t-1}\alpha_k\right)\;.
\end{split}
  \end{equation}
Since $f$ is continuous and $C$ is compact, $f$ has a Lipschitz
constant $\beta>0$ over $C$. Hence, since
$\sum_{k=0}^{t-1}\alpha_k\leq 2\sqrt{t}$,

\begin{align}
\nonumber
\sum_{k=0}^{t-1}\frac{g_k-g_{k+1}}{\alpha_k}
+\gamma D^2\sum_{k=0}^{t-1}\alpha_k
&=
\frac{g_0}{\alpha_0}-\frac{g_t}{\alpha_{t-1}}+
\sum_{k=1}^{t-1}\left(\frac{1}{\alpha_k}-\frac{1}{\alpha_{k-1}}\right)g_k
+\gamma D^2\sum_{k=0}^{t-1}\alpha_t
\\
\nonumber
&\leq
\beta D\left(\frac{1}{\alpha_0}+
\sum_{k=1}^{t-1}\left(\frac{1}{\alpha_k}-\frac{1}{\alpha_{k-1}}\right)
\right)
+2\gamma D^2\sqrt{t}\\
&=(\beta D+2\gamma D^2)\sqrt{t}\;,
\label{e:op1}
\end{align}
and substituting \eqref{e:op1} into \eqref{e:ol1} yields the
result.
\end{proof}

\begin{remark}
\label{rmk:mods}
Slight modifications in the proof of Theorem~\ref{theo:t} also
yield convergence results for the abs-smooth 
version of existing variants of the smooth Frank-Wolfe algorithm.
\begin{enumerate}
\item
The {\em fixed-horizon Frank-Wolfe algorithm} (see
\textup{\cite{CGFWSurvey2022}}),
where a horizon $T\in\NN$ is chosen and, for all iterations, the
fixed step size $\alpha_t\equiv 1/\sqrt{T}$ is used. This
simplifies \eqref{e:op1} and guarantees that the
gap is bounded by $\mathcal{O}(1/\sqrt{T})$ after $T$ iterations --
consistent with the results which require differentiability
\cite{CGFWSurvey2022}.
\item
The {\em monotone Frank-Wolfe algorithm} \textup{\cite{CBS21}},
wherein we keep the
open-loop step sizes $\alpha_t=1/\sqrt{1+t}$ and modify the iterate
to update $x_{t+1}=\alpha_t v_t + (1-\alpha_t)x_t$ if $f(\alpha_t
v_t + (1-\alpha_t)x_t) < f(x_t)$ and otherwise set $x_{t+1}=x_t$.
Simple case analysis reveals that, in both update scenarios,
\eqref{e:ol1} holds, and the proof proceeds identically as before.
\item
For the setting when $\Delta f(x_t;\cdot)$ is convex,
one can derive an abs-smooth version of the {\em short-step
Frank-Wolfe algorithm} (see \textup{\cite{CGFWSurvey2022}}). In the smooth
setting, the step size is
adaptively selected in order to minimize the smoothness inequality
over $[0,1]$; for the abs-smooth setting, the step size 
$\alpha_t=\min\{1,\Delta f(x_t;v_t-x_t)/2\gamma\|v_t-x_t\|^2\}$
adaptively minimizes the upper bound arising from
Theorem~\ref{theo:approx} over the same set.
Just as in the smooth nonconvex case \textup{\cite{P20}}, the function
values
$(f(x_t))_{t\in\nn}$ are guaranteed to monotonically decrease, and
$\mathcal{O}(1/\sqrt{t})$ convergence is acquired. While only local
convexity of $\Delta f(x_t;\cdot)$ is guaranteed in 
general (cf.  Proposition~\ref{prop:dir}), convexity of $\Delta
f(x_t;\cdot)$ is guaranteed in a variety of applications as detailed in the
beginning of the next section.
\end{enumerate}
\end{remark}

If one uses the open-loop step size strategy $\alpha_t=2/(t+2)$, we
have also observed $\mathcal{O}(1/t)$ convergence in experiments on
both convex and nonconvex objectives. This can be proven to hold
under the convexity-type inequality $\Delta f(x_t;x_*-x_t)\leq
f(x_*)-f(x_t)$, which is not guaranteed to hold in the general abs-smooth setting.


\section{The piecewise linear subproblem}
\label{sec:subprob}

This section concerns the central generalization from the linear
subproblem in Line~\ref{fw:2} of Algorithm~\ref{alg:fw} to the
piecewise linear subproblem in Line~\ref{a:s2} of
Algorithm~\ref{alg:asfw}. In particular, for $\bx\in\mathbb{R}^n$,
and $\alpha>0$ we must solve
\begin{equation}
\label{e:pwp}
\underset{v\in C}{\text{min}}\;\;\Delta f(\bx;\alpha(v-\bx))\;.
\end{equation}
While our theoretical analysis demonstrates that one can 
achieve the same per-iteration rate of convergence as in the smooth
setting of Frank-Wolfe, this nonetheless requires solving 
the more challenging subproblem~\eqref{e:pwp}. As can be seen,
instead of a
constrained linear problem, one now has to solve a constrained
piecewise linear problem. This can be a challenging optimization
problem on its own, and there is no off-the-shelf algorithm for its
solution; see \textup{\cite{Baetal20}} for a recent overview of nonsmooth
optimization approaches and \textup{\cite{KrWaGr21}} for the case of
piecewise linear constraints.

In Section~\ref{sec:subprobsolve}, we discuss our approach for
numerically solving \eqref{e:pwp} in the case when $C$ is
polyhedral. Our methodology is guaranteed to find a local
minimizer, hence for the case when $\Delta f(\bx;\cdot)$ is convex,
we solve \eqref{e:pwp} exactly. As will be shown in Section~\ref{sec:num}, all constrained LASSO
problems (in the sense of \textup{\cite{GaKiZh18}}) have convex
piecewise linearizations. This also holds, e.g., for the Mifflin II
problem in Examples~\ref{ex:M2a} and \ref{ex:M2b} and the
counterexample by Nesterov \textup{\cite[Example~1]{Ne18}}.
In practice, we observe convergence to a first-order minimal
solution even without the assumption of convexity. So, we
conjecture that a relaxation of \eqref{e:pwp} may be sufficient to
yield convergence, and we discuss this gap between theory and
observations further in Section~\ref{sec:altsubprob}.

\subsection{Solving the piecewise linear subproblem}
\label{sec:subprobsolve}

In this subsection, we present an approach to minimize piecewise linear functions on a polyhedral feasible set.
To better explain our methodology, we introduce
the abs-linearization in more detail. Due to the smoothness in
Definition~\ref{def:abssmooth}, the following matrices and vectors
are well-defined 
\begin{equation*}
\begin{aligned}
Z &= \frac{\partial}{\partial x} F(x,z,w) \in \rr^{s \times n} \; , \\
M &= \frac{\partial}{\partial z} F(x,z,w) \in \rr^{s \times s} & \textrm{strictly lower triangular} \; , \\
L &= \frac{\partial}{\partial w} F(x,z,w) \in \rr^{s \times s} & \textrm{strictly lower triangular} \; , \\
a & = \frac{\partial}{\partial x} \varphi(x,z) \in \rr^n \; , & b= \frac{\partial}{\partial z} \varphi(x,z) \in \rr^s \; .
\end{aligned}
\end{equation*}
For $f\in\cc_{\abs}^d(\rr^n)$, these matrices define the {\em
abs-linear form} of its piecewise linear model $f_{PL}$ localized
at $\bx\in\rr^n$ that is given by
\begin{align}
  \begin{split}
    \label{eq:alf}
         z &= c + Z \Delta x + Mz+L|z|\;,\\ 
         f_{PL}(\bx ; \Delta x)  &= d + a^\top\! \Delta x + b^\top \!z\;,
  \end{split}
\end{align}
for every $\Delta x\in\mathbb{R}^n$,
where the constants $c\in\R^s$ and $d \in \R$ are chosen
appropriately \textup{\cite{Gr13}}. Such an abs-linear form can be
generated using appropriate variants of AD
\cite{adolc,cppad,tapenade}. The switching variables $z$ in
\eqref{eq:alf}, which technically depend on $\Delta
x\in\mathbb{R}^n$, are used to define the \emph{signature matrix}
of $f_{PL}(\bx;\cdot)$, given by
\begin{align*}
\Sigma(\Delta x) = \diag(\sigma(\Delta x))\;,\quad \textrm{where} \quad \sigma(\Delta x) = \sign{z(\Delta x)} \;.
\end{align*}
For a fixed signature $\sigma\in\{-1,0,1\}^s$ and
$\Sigma=\diag(\sigma)$, the inverse image 
$$ P_\sigma  \; = \;  \{ \Delta x \in \R^n :  \sgn(z(\Delta x)) = \sigma \}   $$
is called a {\em signature domain}. These regions
$(P_\sigma)_{\sigma\in\{-1,0,1\}^s}$ are relatively open polyhedra 
that form a partition of $\R^n$.
It follows from \cite[Proposition~2.2.2]{Sch12} that
each piecewise linear function can be written in an abs-linear form with appropriately-sized
vectors and real lower-triangular matrices. Since
every abs-linear form has a switching variable $z$, it
also gives rise to signature domains.

Note that the minimization of the piecewise linear function $
f_{PL}(\bx ; \cdot)$ is equivalent to the minimization of $\Delta
f(\bx; \cdot-\bx)$, and their abs-linear forms only differ in the
constants $c$ and $d$. Since both functions are affine when
restricted to a particular signature domain $P_{\sigma}$,
constrained minimization over $P_{\sigma}$ is achieved via the
solution of one linear program (provided a solution exists).

Our approach is to adapt the Active Signature Method (ASM) of
\cite{GW18}, which computes an unconstrained local minimizer of
a given piecewise linear function $\psi$ (for our applications, we
will set $\psi=\Delta f(\bx;\alpha(\cdot-\bx))$).  The key idea of
the ASM as proposed in \textup{\cite{GW18}} is to
perform successive linear minimization over the signature domains
of $\psi$.
Initializing on the signature domain $\sigma$ which contains $\bx$,
ASM then computes a minimizer of a regularized strongly convex
problem $\psi(\cdot)+\|\cdot\|_Q^2$ (where $Q\geq 0$) subject to
the constraint $\overline{P_{\sigma}}$. Note that (a) the quadratic
penalty term ensures the existence of a minimizer, and (b) the
solution may have a different signature than $\bx$ if it resides
on the boundary of a signature domain. 
A major feature of the ASM is that, in polynomial time, it can
determine if the constrained minimizer over $P_\sigma$ is a local
minimizer of $\psi$ over $\R^n$, see \cite{GW18}. If local optimality is detected,
ASM terminates. Otherwise, ASM identifies an adjacent polyhedron
$P_{\sigma^+}$ which ensures descent of the target function $\psi$, see \cite{GW18},
in the sense that $$\min_{v\in P_{\sigma^+}}\psi(v)<\min_{v\in
P_{\sigma}}\psi(v)\;,$$ and repeats an iteration. Since there is a
finite number of signature domains, ASM is guaranteed to terminate.

\begin{algorithm}[b]
\caption{Adapted Active Signature Method (AASM)}
\label{alg:aasm}  
\begin{algorithmic}[1]
\REQUIRE Point $\bx\in \mathbb{R}^n$, piecewise linear function
$\psi$
\STATE Initialize $P\leftarrow P_{\sigma(\bx)}$
\FOR{$t=0$ \textbf{to} $\dotsc$}
\STATE Compute $v_*\in
\argmin_{v\in\overline{P}\cap C}\psi(v)$
\label{aasm:2}
\IF{$v_*$ is a local minimizer of $\psi$ over $C$} 
\STATE
Return $v_*$.
\ELSE 
\STATE
$P\leftarrow P_{\sigma^+}$ which guarantees $P_{\sigma^+}\cap
C\neq\varnothing$ and descent of $\psi$
\label{aasm:pplus}
\ENDIF
\ENDFOR
\end{algorithmic}
\end{algorithm}

For our setting, Algorithm~\ref{alg:aasm} performs successive
minimization of $\psi$ over the signature domains intersected with
$C$. Since our feasible domain $C$ is described by linear
equalities and inequalities, they can be added as additional
constraints to the description of the signature domain $P_\sigma$,
effectively encoding the constraint $\overline{P}_{\sigma}\cap C$
as a single polyhedron.
Since $C$ is compact, $\psi$ is guaranteed to possess a minimizer
on every subdomain $\overline{P}_{\sigma}\cap C$ (provided it is
nonempty). Therefore, we can remove the quadratic regularization
term from ASM which was needed to guarantee existence of a
solution. Since $\psi$ is affine on every subdomain, computing a
minimizer in Line~\ref{aasm:2} of  Algorithm~\ref{alg:aasm} is
performed by a single LP call. The LP over $C\cap\overline{P}$ may
(in the worst case) include $s$ more linear inequality constraints
than minimizing over $C$.
Next, Algorithm~\ref{alg:aasm} proceeds as in the ASM by
checking the optimality conditions of \textup{\cite{GW18}} with $Q=0$.
Algorithm~\ref{alg:aasm} terminates if optimality is detected, and
otherwise it proceeds to a new polyhedron which guarantees descent
of $\psi$, see~\cite{GW18}.

Under the assumption that $\Delta f(x_t;\cdot)$ is a convex
function, local minima
coincide with global minima so we solve \eqref{e:pwp} in
Algorithm~\ref{alg:asfw} by applying Algorithm~\ref{alg:aasm} with
$\psi(\cdot)=\Delta f(x_t;\cdot-x_t)$ and $\bx=x_t$. We have also observed
good algorithmic performance and convergence in settings where
$\Delta f(x_t;\cdot)$ is not guaranteed to be convex. The theoretical analysis of this behavior is the subject of future work.

\subsection{Alternative subproblems}
\label{sec:altsubprob}

Most of the computational effort in
Algorithm~\ref{alg:asfw} comes from computing an abs-linearization
(performed once per iteration), and solving the sequence
of linear programs in the inner solver Algorithm~\ref{alg:aasm}.
Our experiments indicate that Algorithm~\ref{alg:aasm} often
terminates after a low number of iterations, and hence it uses a
low number of LP calls in-practice.
Nonetheless, for some specific problems, solving \eqref{e:pwp} may
require Algorithm~\ref{alg:aasm} to visit all signature domains in
a Klee-Minty exhaustive search \textup{\cite{GW18}}, costing one
LP call per signature domain. Since, in worst case settings, the
number of signature domains is exponential in the number of
switching variables $s$, the theoretical upper-bound
on the number of LP calls per iteration of
Algorithm~\ref{alg:asfw} is quite high. Since this does not reflect
what we see is required in-practice, this begs the question,
{\em does Algorithm~\ref{alg:asfw} produce a first-order minimal
solution of \eqref{e:p} when, in Line~\ref{a:s2}, $v_t$ is only a
partial solution to \eqref{e:pwp}?} 
In the smooth setting, a similar question was answered in the
affirmative via ``lazified'' variants of the Frank-Wolfe
Algorithm~\ref{alg:fw} \cite{BPZ17}.

We provide a partial negative answer to this question by
establishing that one iteration of the inner solver
Algorithm~\ref{alg:aasm} is insufficient to find a solution. This
is demonstrated by considering an abs-linear objective function
$f$. In this setting, at every iteration of
Algorithm~\ref{alg:asfw}, $\Delta f(\bx;\cdot)$ is equal to
$f(\cdot)$ (modulo translation), and hence the signature domains
of the abs-linear model remain unchanged.
Note that, if we only use one inner iteration of
Algorithm~\ref{alg:aasm}, every vertex $(v_t)_{t\in\mathbb{N}}$ of
Algorithm~\ref{alg:asfw} resides in the closure of the same
signature domain, i.e., the initial signature domain. Therefore,
the iterates $(x_t)_{t\in\mathbb{N}}$ will also remain in the
same closed convex set.
So, unless Algorithm~\ref{alg:asfw} is initialized on
a signature domain whose closure contains a first-order minimal
solution of \eqref{e:p}, it cannot yield a solution.

Interestingly, even though replacing Line~\ref{a:s2} in
Algorithm~\ref{alg:asfw} with one iteration of
Algorithm~\ref{alg:aasm} will not yield a solution in general, this
algorithm {\em is} sufficient to guarantee that
$\Delta f(x_t;\alpha_t(v_t-x_t))/\alpha_t$ converges to zero!
Indeed, the proof of
Theorem~\ref{theo:t} only relies on Theorem~\ref{theo:approx} and
negativity of our generalized gap, which we point out below.
\begin{lemma}
\label{l:s1m}
Let $\alpha>0$, let $\bx\in\mathbb{R}^n$, and let
$f\in\cc_{\abs}^d(\rr^n)$. After one iteration of
Algorithm~\ref{alg:aasm} with $\psi=\Delta
f(\bx;\alpha(\cdot-\bx))$ and $x_0=\bx$, we have
\begin{equation*}
\Delta f(\bx;\alpha(v_1-\bx))\leq 0\;.
\end{equation*}
\end{lemma}
\begin{proof}
Since $v_1$ is the result of minimizing over
$\overline{P_{\sigma}}\ni\bx$, we know
$\Delta f(\bx;\alpha(v_*-\bx))\leq \Delta
f(\bx;\alpha(\bx-\bx))=0$.
\end{proof}
In view of Lemma~\ref{l:s1m}, one could use the same proof
technique in Theorem~\ref{theo:t} to show that $\Delta
f(x_t;\alpha_t(v_t-x_t))/\alpha_t$ converges to zero. However,
since $v_t$ is no longer a solution to \eqref{e:fom2}, $\Delta
f(x_t;\alpha_t(v_t-x_t))/\alpha_t$ is not our generalized
Frank-Wolfe gap, so we can not infer first-order minimality via Theorem~\ref{theo:fom1}.

In all of our experiments, we have observed convergence to a
solution even when $\Delta f(\bx;\cdot)$ was not guaranteed to be
convex. Therefore, since Algorithm~\ref{alg:aasm} terminates when
it detects local optimality, we conjecture that
Algorithm~\ref{alg:asfw} will still yield a first-order minimal
solution of \eqref{e:p} when provided a locally minimal solution to
\eqref{e:pwp}.

\section{Numerical examples}
\label{sec:num}
To verify our theoretical findings, we implemented the Frank-Wolfe
approach for abs-smooth functions as stated in
Algorithm~\ref{alg:asfw} in C++. In Section~\ref{subsec:RN}, we
benchmark our subproblem solver Algorithm~\ref{alg:aasm}. In
Section~\ref{subsec:tests}, we test our full algorithm on a suite
of scalable problems from nonsmooth optimization, and particular
attention is given to constrained LASSO problems in
Section~\ref{sec:lasso}.
For the generation of the local piecewise
linear model with
abs-linearization we used ADOL-C \cite{adolc}. For Linear
Programming, we employed the solver HiGHS \cite{HuHa18}.

\subsection{Algorithm~\ref{alg:aasm} (AASM): Rosenbrock-Nesterov II}
\label{subsec:RN}
In this section, we analyze the performance of our inner solver
AASM on the Rosenbrock-Nesterov II function. According to
\textup{\cite{GuOv12}}, Nesterov suggested the Rosenbrock-like test
function defined as
  \begin{align}
 \psi : \R^{n} \rightarrow \R, \quad \psi(x)=\, &
\tfrac{1}{4}  \; |x_1-1|\; + \, \sum_{i=1}^{n-1} \left  |x_{i+1}
-2|x_i|+1 \right |\;,  \label{secondrosen}
  \end{align}
  which is piecewise linear and nonconvex. The analysis presented
in \textup{\cite{GuOv12}} shows that this test function 
has (a) a unique global minimizer $x^*=(1,1, \ldots, 1)\in\R^n$ and 
(b) $2^{n-1}-1$ other stationary points which are not local minima,
at which nonsmooth optimization algorithms may get stuck.
 Numerical tests showed that with the initial point
   \begin{align*}
     x_{0,1}= -1,\qquad x_{0,i} = 1\quad 2\le i \le n\;,
  \end{align*}
   it is very likely to encounter all of the stationary points; see
   again \textup{\cite{GuOv12}}. This also explains the high number
of iterations
   required by AASM since it visits all the stationary points and
verifies the non-optimality of all of them, except for the last one.
One has to note that for the examples solved with our Frank-Wolfe
methods the number of iterations in each inner solve is in almost
all cases less than five.
   
   Comparisons of three different
solvers, i.e., a bundle method, an adapted quasi-Newton method, and
ASM were presented for $1\le n \le 10$ in
\textup{\cite{GW18}}.
It was found that the bundle method and the adapted quasi-Newton
method got trapped at one of the stationary points for $3\le n \le
10$ after a high number of iterations. That is, they were not able
to find the minimizer. On the other hand, for all
values of $n$ considered, ASM reached the
minimizer exactly after $2^{n}$ iterations. Larger values of $n$
were not considered due to numerical difficulties for these higher
dimensions.

   To validate the proposed Frank-Wolfe algorithm for abs-smooth functions, we  introduced artificial bounds on the variables such that we have a compact convex feasible set. That is, we consider
      \begin{align*}
     C = \{ x \in \R^n \,|\,-20 \le x_{i} \le 20,\;  1\le i \le n\} \;.
      \end{align*}
Note that the constraint $C$ excludes neither the suboptimal
stationary points nor the global minimizer.

We coded the function evaluation exactly as stated in \eqref{secondrosen}. That is,
the required abs-linear form \eqref{eq:alf} was generated
by ADOL-C. When applying Algorithm~\ref{alg:aasm} to minimize $\psi$ as defined in
\eqref{secondrosen}, we obtain the behavior shown in
Table~\ref{tab:rosnes}, where we also state the number of existing
polyhedra for $n\le 10$ but skip this number for $n>10$ due to the
very large value. As compared to ASM which
took $2^n$ iterations \cite{GW18}, the iteration numbers reduced
significantly to $2^{n-1}$ corresponding exactly to the number of
stationary points. Hence, due to the LP solve, AASM actually visits a very small fraction of all existing polyhedra
(cf. Table~\ref{tab:rosnes}). Furthermore, the dimension $n$ could
be increased
considerably more than in \textup{\cite{GW18}}. It is very interesting to
note that the presolve of
the linear solver HiGHS reduces the linear optimization problem on
each polyhedron to an empty one such that a solution could be
computed without performing a simplex step at all -- as seen in
Table~\ref{tab:rosnes}. This is a tremendous advantage in
comparison to ASM which can be seen as an adapted QP solver.
We
observe this behavior of the linear solver also for other test
problems. Due to this fact, even the largest instance
which required more than 500,000 iterations could be solved
in less than three minutes.
  \begin{table}[ht]
     \begin{center}
    \begin{tabular}{|l|c|c|c|c|c|c|c|c|c|c|} \hline
      \qquad\quad $n$        \!&\multicolumn{1}{c|}{1}&\multicolumn{1}{c|}{2}&\multicolumn{1}{c|}{3}&\multicolumn{1}{c|}{4}&\multicolumn{1}{c|}{5}&\multicolumn{1}{c|}{6}&\multicolumn{1}{c|}{7}&\multicolumn{1}{c|}{8}&\multicolumn{1}{c|}{9}&\multicolumn{1}{c|}{10}\\ \hline
      \# polyhedra   & 1 & 8 & 32 & 128 & 512 & 2048 & 8192 & 32768 & 131072 & 524288 \\
      \# iter (AASM) & 1 & 2 & 4 & 8 & 16 & 32 & 64 & 128 & 256 & 512\\
      \# iter (simplex) & 0 & 0 & 0 & 0 & 0 & 0 & 0 & 0 & 0 & 0 \\\hline \hline
      \qquad\quad $n$        \!&\multicolumn{1}{c|}{11}&\multicolumn{1}{c|}{12}&\multicolumn{1}{c|}{13}&\multicolumn{1}{c|}{14}&\multicolumn{1}{c|}{15}&\multicolumn{1}{c|}{16}&\multicolumn{1}{c|}{17}&\multicolumn{1}{c|}{18}&\multicolumn{1}{c|}{19}&\multicolumn{1}{c|}{20}\\ \hline
      \# iter (AASM) &  1024 & 2048 & 4096 & 8192 & 16384 &  32768 & 65536 &   131072 &  262144 & 524288 \\
      \# iter (simplex) & 0 & 0 & 0 & 0 & 0 & 0 & 0 & 0 & 0 & 0 \\ \hline \end{tabular}
     \end{center}
	 \caption{Number of signature domains and iteration counts for
Rosenbrock-Nesterov II example and Algorithm~\ref{alg:aasm}, i.e.,
the adapted active signature method (AASM).}
     \label{tab:rosnes}
   \end{table}

\subsection{Algorithm~\ref{alg:asfw}: Standard nonsmooth problems}
\label{subsec:tests}

We tested our abs-smooth Frank-Wolfe (ASFW)
Algorithm~\ref{alg:asfw} on several standardized test cases (for
details, see
\textup{\cite{BaKaMae14}}). For that purpose, we
implemented the two convex examples MAXQ and Chained LQ as well as
the four nonconvex examples: Number of active faces, Chained
Mifflin 2, Chained Crescent 1, and Chained Crescent 2.
Furthermore, we tested the first nonsmooth
and nonconvex Rosenbrock-Nesterov function analyzed in
\cite{GuOv12}. For all benchmark objectives except for MAXQ, we add bounds 
\begin{align*}
C = \{ x \in \R^n\,|\, -5 \le x_i \le 5 , 1 \le i \le n\}\;,
\end{align*}
which do not interfere with the optimal solution, allowing us to
properly test our implementation.
Section~\ref{sec:maxq} reports on the results for MAXQ and
considers different feasible sets in which constraints are active at the
solution. All example functions are scalable such that we could
examine various dimensions $n$. For all test cases, we observe a
very similar convergence behavior for the
step size $\alpha_t=1/\sqrt{1+t}$, namely the convergence rate
$\cO(1/\sqrt{t})$.

Since the results are very similar for all test cases, we
illustrate the convergence behavior just for two convex examples
(Sections~\ref{sec:maxq} and \ref{sec:clq}) and two nonconvex
examples (Sections~\ref{sec:nRN} and \ref{sec:nCC1}).

\subsubsection{MAXQ problem}
\label{sec:maxq}
The MAXQ problem is given by
\begin{align}
 f: \R^n \mapsto \R, f(x)& = \max_{1\leq i\leq n} x_i^2 \label{eq:maxq}\\
(\forall i\in\{1,2,\ldots,n\})\quad
(x_0)_i&=\begin{cases}
\phantom{-}i&\text{if}\;\;i\in\{1,\ldots,\lfloor n/2\rfloor \} \; ,\\
-i&\text{if}\;\;i\in\{\lfloor n/2\rfloor +1,\ldots,n\}\;.
\end{cases} \nonumber
 \end{align}
For this academic test case, we also consider the following feasible sets:
\begin{align*}
C_1  & = \left\{ x \in \R^n\,|\, -5 \le x_i \le 2i-2 \mbox{ for }
i\in\{1,\ldots,\lfloor n/2\rfloor \};\quad -2i+2\le x_i \le 5  \mbox{ for }
i\in\{\lfloor n/2\rfloor +1,...,n\}\right\}\;,\\
C_2 & = \left\{ x \in \R^n\,|\, \phantom{-}0\le x_i \le
2i-2\mbox{ for } i\in\{1,\ldots,\lfloor n/2\rfloor \};\quad
-2i+2\le x_i \le 0  \mbox{ for }
i\in\{\lfloor n/2\rfloor +1,\ldots,n\}\right\}\;,\\
C_3 & = \left\{ x \in \R^n\,|\, \phantom{-}1\le x_i \le 2i-1
\mbox{ for } i\in\{1,\ldots,\lfloor n/2\rfloor \};\quad -2i+1\le x_i \le -1  \mbox{
for } i\in\{\lfloor n/2\rfloor +1,\ldots,n\}\right\}\;.
\end{align*}
The constraints in $C_1$ are inactive at the global solution $x_* =
0\in \R^n$, while constraints in $C_2$ are active precisely at
$x_*$. For $C_3$ we obtain the new optimal solution whose $i$th
component is given by
\begin{equation*}
 (x_*)_{i} =\begin{cases}
\phantom{-}1 &\text{for}\;\; i\in\{1,\ldots,\lfloor n/2\rfloor \} \; ,\\
-1&\text{for}\;\; i\in\{\lfloor n/2\rfloor +1,\ldots,n\}\;.
\end{cases}
\end{equation*}  
Once more, we coded the function evaluation exactly as stated in \eqref{eq:maxq}. That is,
the abs-smooth form was used only internally in ASFW to generate the numerical results. For the step size $\alpha_t=1/\sqrt{1+t}$, the observed convergence
history is shown in Figure~\ref{fig:maxq}. Algorithm~\ref{alg:asfw}
terminated regularly with a norm of the generalized Frank-Wolf gap
being smaller than $10^{-10}$. For all combinations
considered here, the function values also converged to the optimal
value.

\begin{figure}[t]
\begin{center}
\subfloat[][feasible set $C_1$]{
\centering
\begin{tikzpicture}[scale=1.0]
\begin{axis}[height=5.5cm,width=6.5cm, legend cell align={left},
ylabel style={yshift=-0.75em,font=\scriptsize}, 
xlabel style={yshift=0.75em,font=\scriptsize},
tick label style={font=\tiny},
legend entries={$n=20$,$n=10$,$n=5$},
legend style={font=\tiny},
xlabel=Iteration $t$, 
ylabel=$-\Delta f(x_t;\alpha_t(v_t-x_t))/\alpha_t$,
grid,
xmin =0, xmax=60, ymin=1e-10, ymax=1e3, ymode=log, mark repeat = 8]
\addplot+[thick, mark=triangle, color=bred] table[x={iter}, y={n_20}]
{matlab/iter_20_sqrt.txt};
\addplot+[thick, mark=square, color=bblu] table[x={iter}, y={n_10}]
{matlab/iter_20_sqrt.txt};
\addplot+[thick, mark=o, color=black] table[x={iter}, y={n_5}]
{matlab/iter_20_sqrt.txt};
\end{axis}
\end{tikzpicture}
}
\quad
\subfloat[][feasible set $C_2$]{
\begin{tikzpicture}[scale=1.0]
\begin{axis}[height=5.5cm,width=6.5cm, legend cell align={left},
label style={font=\tiny},
ylabel style={yshift=-0.75em,font=\scriptsize}, 
xlabel style={yshift=0.75em,font=\scriptsize},
tick label style={font=\tiny},
legend entries={$n=20$,$n=10$,$n=5$},
legend style={font=\tiny},
xlabel=Iteration $t$, 
ylabel=$-\Delta f(x_t;\alpha_t(v_t-x_t))/\alpha_t$,
grid,
xmin =0, xmax=80, ymin=1e-10, ymax=1e3, ymode=log, mark repeat = 8]
\addplot [thick, mark=o, color=black] table[x={iter}, y={n_20}]
{matlab/iter_24_n_20_sqrt.txt};
\addplot [thick, mark=square, color=bblu] table[x={iter}, y={n_10}]
{matlab/iter_24_n_10_sqrt.txt};
\addplot [thick, mark=triangle, color=bred] table[x={iter}, y={n_5}]
{matlab/iter_24_n_5_sqrt.txt};
\end{axis}
\end{tikzpicture}
}
\quad
\subfloat[][feasible set $C_3$]{
\begin{tikzpicture}[scale=1.0]
\begin{axis}[height=5.5cm,width=6.5cm, legend cell align={left},
label style={font=\tiny},
ylabel style={yshift=-0.75em,font=\scriptsize}, 
xlabel style={yshift=0.75em,font=\scriptsize},
tick label style={font=\tiny},
legend entries={$n=20$,$n=10$,$n=5$},
legend style={font=\tiny},
xlabel=Iteration $t$, 
ylabel=$-\Delta f(x_t;\alpha_t(v_t-x_t))/\alpha_t$,
grid,
xmin =0, xmax=100, ymin=1e-7, ymax=1e3, ymode=log, mark repeat = 10]
\addplot [thick, mark=o, color=black] table[x={iter}, y={n_20}]
{matlab/iter_25_n_20_sqrt.txt};
\addplot [thick, mark=square, color=bblu] table[x={iter}, y={n_10}]
{matlab/iter_25_n_10_sqrt.txt};
\addplot [thick, mark=triangle, color=bred] table[x={iter}, y={n_5}]
{matlab/iter_25_n_5_sqrt.txt};
\end{axis}
\end{tikzpicture}
}
\end{center}
\caption{Generalized Frank-Wolfe gap
$-\Delta f(x_t;\alpha_t(v_t-x_t))/\alpha_t$ versus iteration count
$t$ displaying the convergence behavior of Algorithm~\ref{alg:asfw}
with $\alpha_t=1/\sqrt{1+t}$ on the MAXQ problem
(Section~\ref{sec:maxq}) for various values of $n$.}
\label{fig:maxq}
\end{figure}
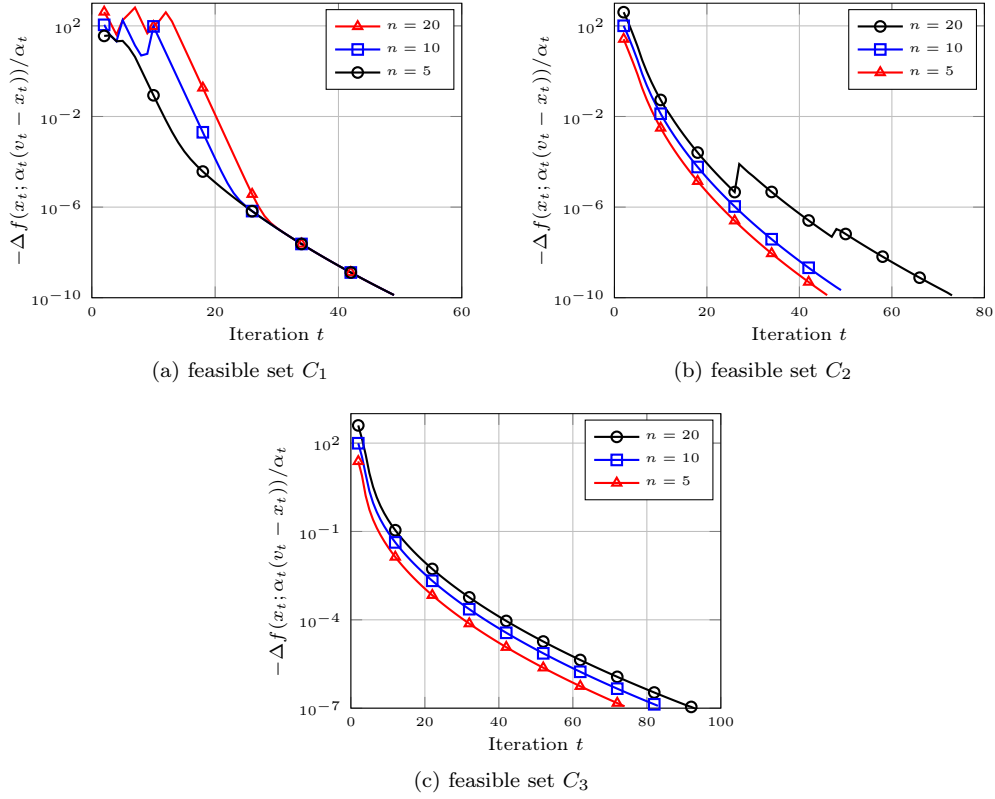

\subsubsection{Chained LQ problem}
\label{sec:clq}
The convex Chained LQ objective problem is given by
\begin{align}
  f(x)& = \sum_{i=1}^{n-1} \max\left\{-x_i -x_{i+1},\; -x_i-x_{i+1}+x_i^2 + x_{i+1}^2-1\right\} \label{eq:chainedLQ1}\\
      & = \sum_{i=1}^{n-1} \tfrac{1}{2}\left(-2x_i -2x_{i+1}+ x_i^2 + x_{i+1}^2-1+\left|x_i^2 + x_{i+1}^2-1\right|\right) \; ,\label{eq:chainedLQ2}\\
     x_{0,i}& = -0.5 \quad \mbox{for all }\;  i=1,...,n\;, \nonumber
\end{align}
where \eqref{eq:chainedLQ1} is the version usually stated for this test problem and also used here for the implementation, whereas \eqref{eq:chainedLQ2} can be used to derive a  corresponding abs-smooth form. For a fixed $\bx$, the  abs-linearization generated in an automated fashion by ADOL-C is given by
\begin{align*}
  \Delta f(\bx,\Delta x)
  & = \sum_{i=1}^{n-1}\Big(-\Delta x_i -\Delta x_{i+1}+ \bx_i\Delta x_i + \bx_{i+1}\Delta x_{i+1}\\
  & \qquad \qquad  +\tfrac{1}{2}\big|\bx_i^2 + \bx_{i+1}^2-1+2\bx_i\Delta x_i + 2\bx_{i+1}\Delta x_{i+1}\big|-\tfrac{1}{2}\big|\bx_i^2 + \bx_{i+1}^2-1\big|\Big)\;.
\end{align*}
As can be seen, this function is convex in $\Delta x$. Hence, Algorithm~\ref{alg:aasm} indeed solves \eqref{e:pwp} globally such that our convergence theory holds. 
For different values of $n$, the convergence history for the first 500 iterates is illustrated in Figure~\ref{fig:CLQ_st}.
One can clearly observe the convergence rate $\cO(1/\sqrt{t})$,
where the rate coefficients vary with $n$, as can be seen from the
proof of Theorem~\ref{theo:t}. Furthermore, we also observe
convergence of the function value to the optimal quantity 
$-(n-1)\sqrt{2}$ for all considered dimensions $n$.
\begin{figure}[t]
\begin{center}             
\begin{tikzpicture}[scale=1.0]
\begin{axis}[height=7cm,width=12.0cm, 
label style={font=\scriptsize},
ylabel style={yshift=-0.85em}, 
xlabel style={yshift=0.75em},
tick label style={font=\tiny},
legend style={font=\tiny},
legend entries={
$c/\sqrt{t}$,
$n=100$,
$n=20$,
$n=5$,
},
xmin =0, xmax=500, ymin=1e-0, ymax=3e4, ymode=log, mark repeat =
100,
xlabel=Iteration $t$, 
ylabel=$-\Delta f(x_t;\alpha_t(v_t-x_t))/\alpha_t$,
grid
]
\addplot [thick, mark=o, mark size=4pt, color=black]
table[x={iter}, y={sqrt}] {matlab/iter_21_sqrt.txt};
\addplot [thick, mark=square, mark size=4pt, color=bblu]
table[x={iter}, y={n_100}] {matlab/iter_21_sqrt.txt};
\addplot [thick, mark=triangle, mark size=4pt, color=bred]
table[x={iter}, y={n_20}] {matlab/iter_21_sqrt.txt};
\addplot [thick, mark=x, mark size=4pt, color=bcyan]
table[x={iter}, y={n_5}] {matlab/iter_21_sqrt.txt};
\end{axis}
\end{tikzpicture}
\end{center}
\caption{Convergence behavior of Algorithm~\ref{alg:asfw} with
$\alpha_t=1/\sqrt{1+t}$ 
on the Chained LQ problem (Section~\ref{sec:clq}) in various
dimensions $n$.}
\label{fig:CLQ_st}
\end{figure}
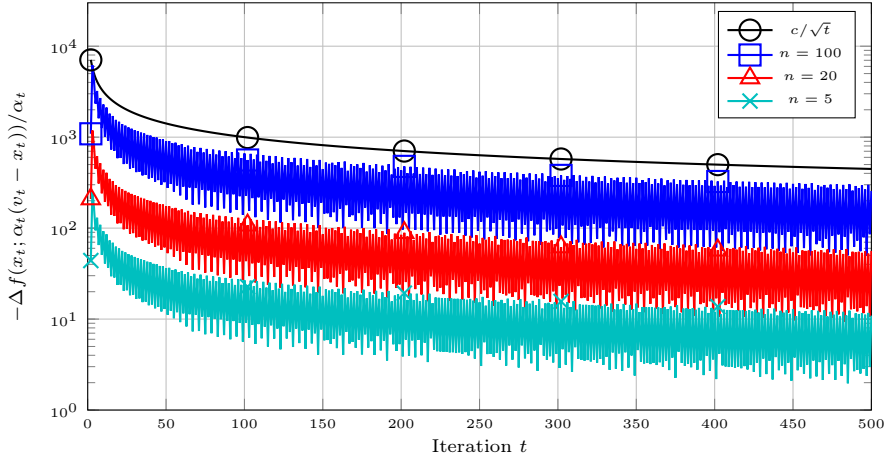

\subsubsection{Nonconvex Rosenbrock-Nesterov problem}
\label{sec:nRN}
Next, we present the results for the nonconvex Rosenbrock-Nesterov
function
 \begin{align*}
  f(x) & =\tfrac{1}{4}(x_1-1)^2 +\sum_{i=1}^{n-1}\left|x_{i+1}-2x_i^2 +1\right| \; ,\\
\quad
  (x_0)_i&=\begin{cases}
-0.5 &\text{if}\; \mod (i,2)=1 \\
\phantom{-}0.5&\text{if}\;  \mod (i,2)=0\;
\end{cases}\qquad \mbox{for}\quad i\in\{1,\ldots,n\}\;.
 \end{align*}
 For different values of $n$, the convergence history for the first 500 iterates  is illustrated in Figure~\ref{fig:ros_st}. Once more, the
convergence rate $\cO(1/\sqrt{t})$ is clearly visible and the
heights of the lines vary with $n$, which is consistent with the
prefactor's dependence on the set diameter $D$ in the proof of
Theorem~\ref{theo:t}. We again observe convergence of the function
value to the optimal value $0$ for all considered dimensions $n$.
\begin{figure}[t]
\begin{center}             
\begin{tikzpicture}[scale=1.0]
\begin{axis}[height=7cm,width=12.0cm, 
label style={font=\scriptsize},
ylabel style={yshift=-0.5em}, 
xlabel style={yshift=0.75em},
tick label style={font=\tiny},
legend style={font=\tiny},
legend entries={
$c/\sqrt{t}$,
$n=100$,
$n=20$,
$n=5$,
},
xmin =0, xmax=2000, ymin=5e-4, ymax=1e3, ymode=log, mark repeat =
200,
xlabel=Iteration $t$, 
ylabel=$-\Delta f(x_t;\alpha_t(v_t-x_t))/\alpha_t$,
grid
]
\addplot [thick, mark=o, mark size=3pt, color=black]
table[x={iter}, y={sqrt}] {matlab/iter_31_sqrt.txt};
\addplot [thick, mark=square, mark size=3pt, color=bblu]
table[x={iter}, y={n_100}] {matlab/iter_31_sqrt.txt};
\addplot [thick, mark=triangle, mark size=3pt, color=bred]
table[x={iter}, y={n_20}] {matlab/iter_31_sqrt.txt};
\addplot [thick, mark=x, mark size=3pt, color=bcyan]
table[x={iter}, y={n_5}] {matlab/iter_31_sqrt.txt};
\end{axis}
\end{tikzpicture}
\end{center}
\caption{Convergence behavior of Algorithm~\ref{alg:asfw} with
$\alpha_t=1/\sqrt{1+t}$ 
on the first Rosenbrock-Nesterov problem (Section~\ref{sec:nRN}) in
various dimensions $n$.}
\label{fig:ros_st}
\end{figure}
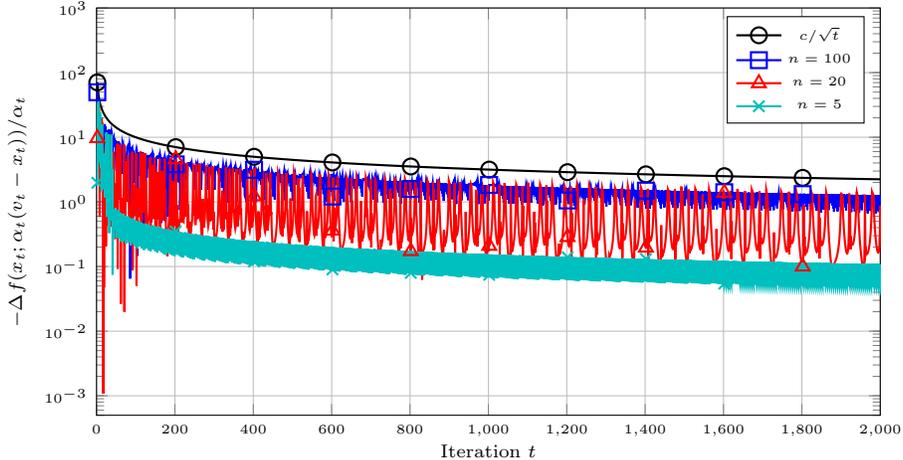

\subsubsection{Nonconvex Chained Crescent 1 problem}
\label{sec:nCC1}
The nonconvex Chained Crescent 1 problem is given by
\begin{align*}
  & f(x)= \max \left\{f_1(x),f_2(x)\right\},&\\
  & \mbox{with}\quad f_1(x)=\sum_{i=1}^{n-1}\left(x_i^2 + (x_{i+1}
-1)^2 +x_{i+1} -1\right),\quad f_2(x)=\sum_{i=1}^{n-1}\left(-x_i^2
- (x_{i+1} -1)^2 +x_{i+1} +1\right) \; ,\\
& (x_0)_i=\begin{cases}
-1.5 &\text{if}\; \mod (i,2)=1\\
\phantom{-}2.0&\text{if}\;  \mod (i,2)=0\;
\end{cases}\qquad \mbox{for}\quad i\in\{1,\ldots,n\}\;.
 \end{align*}
For this problem, we tested the step size  $\alpha_t=2/(t+2)$
since, according to the theory developed for Frank-Wolfe methods,
one may expect to observe a convergence rate of $\cO(1/t)$. We can
verify this convergence behavior numerically; see
Figure~\ref{fig:CC1_t} for the convergence history of the first
500 iterates. However, it is currently unknown if one can always
expect this convergence rate.
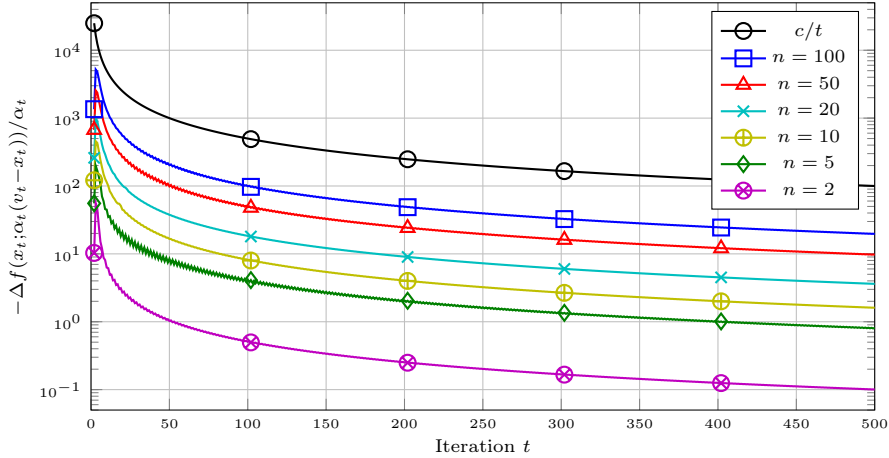
\begin{figure}[ht]
\begin{center}             
\begin{tikzpicture}[scale=1.0]
\begin{axis}[height=7cm,width=12.0cm, 
ylabel style={yshift=-0.75em}, 
xlabel style={font=\scriptsize, yshift=0.75em},
tick label style={font=\tiny},
legend style={font=\scriptsize},
legend entries={
$c/t$,
$n=100$,
$n=50$,
$n=20$,
$n=10$,
$n=5$,
$n=2$,
},
xmin =0, xmax=500, ymin=5e-2, ymax=5e4, ymode=log, mark repeat =
100,
xlabel=Iteration $t$, 
ylabel=$\scriptstyle{-\Delta f(x_t;\alpha_t(v_t-x_t))/\alpha_t}$,
grid
]
\addplot [thick, mark=o, mark size=3pt, color=black]
table[x={iter}, y={t}] {matlab/iter_34_t.txt};
\addplot [thick, mark=square, mark size=3pt, color=bblu]
table[x={iter}, y={n_100}] {matlab/iter_34_t.txt};
\addplot [thick, mark=triangle, mark size=3pt, color=bred]
 table[x={iter}, y={n_50}] {matlab/iter_34_t.txt};
\addplot [thick, mark=x, mark size=3pt, color=bcyan]
table[x={iter}, y={n_20}] {matlab/iter_34_t.txt};
\addplot [thick, mark=oplus, mark size=3pt, color=byellow]
table[x={iter}, y={n_10}] {matlab/iter_34_t.txt};
\addplot [thick, mark=diamond, mark size=3pt, color=bgreen]
table[x={iter}, y={n_5}] {matlab/iter_34_t.txt};
\addplot [thick, mark=otimes,mark size=3pt, color=bpurp]
table[x={iter}, y={n_2}] {matlab/iter_34_t.txt};
\end{axis}
\end{tikzpicture}
\end{center}
\caption{Convergence behavior of Algorithm~\ref{alg:asfw} with
step sizes $\alpha_t=2/(t+2)$
on the Chained Crescent 1 problem (Section~\ref{sec:nCC1}) in
various dimensions $n$.}
\label{fig:CC1_t}
\end{figure}

\subsection{Constrained LASSO problems}
\label{sec:lasso}

Finally, we consider the constrained LASSO problem
\begin{align}
  \min\limits_{x \in \R^n} \; &\tfrac{1}{2} \| A x - y\|^2_2   +  \rho \|x\|_1  \label{eq:lasso}\\
  \mbox{s.t.}\quad  & Bx = b\quad\text{ and } \quad Mx \le d\;, \nonumber
\end{align}  
where $y \in \R^p$ is the response vector, $A\in  \R^{p\times n}$ is the design
matrix, $x \in \R^n$ is the vector of unknown regression coefficients, and
$\rho \ge 0$ is a regularization parameter. The matrices $B\in
\R^{r \times n}$, $M\in \R^{q \times n}$ and the vectors $b\in
\R^r$, $d\in \R^q$ describe additional equality and inequality
constraints. As
its name suggests, the constrained LASSO augments the standard
LASSO \textup{\cite{Ti96}} with additional constraints that allow to take
prior knowledge into account.  This could be, for example, an
ordering of the regression coefficients leading to an ordered LASSO
problem or the requirement of positive regression coefficients
yielding the positive LASSO.

The abs-linearization of the objective function \eqref{eq:lasso} given by
\begin{align*}
  \Delta f(\bx,\Delta x) = (\bx^\top A^\top  - y)A\Delta x  +  \rho \left(\|\bx+\Delta x\|_1-\|\bx\|_1\right)  
\end{align*}
is convex in $\Delta x$. While this fact is not exploited in the
convergence analysis of Theorem~\ref{theo:t}, this does mean that
our subproblem solver Algorithm~\ref{alg:aasm} will yield a global
minimizer.

To test the performance of our algorithm, we use random data
so that we can scale the dimension arbitrarily. The entries of the
design matrix $A$ and the response are generated as independent and
identical standard normal variables. We examine two variants of
\eqref{eq:lasso}. First, we add again the bound constraints
\begin{align*}
C = \{ x \in \R^n \,|\,-5 \le x_{i} \le 5,\;  1\le i \le n\}\;.
\end{align*}
 Hence, for $n$ and $p$ large enough
one expects $0\in\R^n$ as the optimal solution. 
As initial point, we use  a randomly generated vector with entries
which are identically standard normal distributed. 

For the step size $\alpha_t=1/\sqrt{1+t}$, the results using various values of $n$ and $p$ are shown in
Figure~\ref{fig:bclasso}. All optimizations terminated at the optimal solution
with the criterion $\Delta f(x_t,\alpha_t(v_t-x_t))=0$, i.e., when
a first-order minimal point was found (cf. Theorem~\ref{theo:fom1}).
For several combinations of $n$ and $p$, Algorithm~\ref{alg:asfw}
reached a first-order minimal point early and therefore the
corresponding curve ends. Once more, the convergence rate of
$\cO(1/\sqrt{t})$ proven in Theorem~\ref{theo:t} is clearly visible
in all scenarios. 

\begin{figure}[t]
\begin{center}             
\begin{tabular}{c c}
\subfloat[][$n=125$]{
\begin{tikzpicture}[scale=1.0]
\begin{axis}[height=5.5cm,width=6.5cm, legend pos=north east,
ylabel style={yshift=-0.75em},
xlabel style={yshift=0.75em,font=\scriptsize},
tick label style={font=\tiny},
legend style={font=\tiny},
legend entries={
$c/\sqrt{t}$,
$p=500$,
$p=375$,
$p=250$,
},
xmin =0, xmax=500, ymin=1e1, ymax=5e5, ymode=log, mark repeat =
60,
xlabel=Iteration $t$, 
ylabel=$\scriptstyle{-\Delta f(x_t;\alpha_t(v_t-x_t))/\alpha_t}$,
grid
]
\addplot [thick, mark=o, mark size=2.5pt, color=black]
table[x={n}, y={sqrt}] {matlab/iter_lasso_125_sqrt.txt};
\addplot [thin, mark=square, mark size=2.5pt, color=bblu]
table[x={n}, y={p_500}] {matlab/iter_lasso_500_125_sqrt.txt};
\addplot [thin, mark=triangle, mark size=2.5pt, color=bred]
table[x={n}, y={p_375}] {matlab/iter_lasso_375_125_sqrt.txt};
\addplot [thin, mark=otimes, mark size=2.5pt, color=bcyan]
table[x={n}, y={p_250}] {matlab/iter_lasso_250_125_sqrt.txt};
\end{axis}
\end{tikzpicture}
}
&
\subfloat[][$n=250$]{
\begin{tikzpicture}[scale=1.0]
\begin{axis}[height=5.5cm,width=6.5cm, legend pos=north east,
ylabel style={yshift=-0.75em},
xlabel style={yshift=0.75em,font=\scriptsize},
tick label style={font=\tiny},
legend style={font=\tiny},
legend entries={
$c/\sqrt{t}$,
$p=1000$,
$p=750$,
$p=500$,
},
xmin =0, xmax=500, ymin=1e1, ymax=5e5, ymode=log, mark repeat =
60,
xlabel=Iteration $t$, 
ylabel=$\scriptstyle{-\Delta f(x_t;\alpha_t(v_t-x_t))/\alpha_t}$,
grid
]
\addplot [thick, mark=o, mark size=2.5pt, color=black]
table[x={n}, y={sqrt}] {matlab/iter_lasso_250_sqrt.txt};
\addplot [thin, mark=square, mark size=2.5pt, color=bblu]
table[x={n}, y={p_1000}] {matlab/iter_lasso_1000_250_sqrt.txt};
\addplot [thin, mark=triangle, mark size=2.5pt, color=bred]
table[x={n}, y={p_750}] {matlab/iter_lasso_750_250_sqrt.txt};
\addplot [thin, mark=otimes, mark size=2.5pt, color=bcyan]
table[x={n}, y={p_500}] {matlab/iter_lasso_500_250_sqrt.txt};
\end{axis}
\end{tikzpicture}
}
\\ 
\subfloat[][$n=500$]{
\begin{tikzpicture}[scale=1.0]
\begin{axis}[height=5.5cm,width=6.5cm, legend pos=north east,
ylabel style={yshift=-0.75em},
xlabel style={yshift=0.75em,font=\scriptsize},
tick label style={font=\tiny},
legend style={font=\tiny},
legend entries={
$c/\sqrt{t}$,
$p=2000$,
$p=1500$,
$p=1000$,
},
xmin =0, xmax=800, ymin=1e1, ymax=5e6, ymode=log, mark repeat =
140,
xlabel=Iteration $t$, 
ylabel=$\scriptstyle{-\Delta f(x_t;\alpha_t(v_t-x_t))/\alpha_t}$,
grid
]
\addplot [thick, mark=o, mark size=2.5pt, color=black]
table[x={n}, y={sqrt}] {matlab/iter_lasso_500_sqrt.txt};
\addplot [thin, mark=square, mark size=2.5pt, color=bblu]
table[x={n}, y={p_2000}] {matlab/iter_lasso_2000_500_sqrt.txt};
\addplot [thin, mark=triangle, mark size=2.5pt, color=bred]
table[x={n}, y={p_1500}] {matlab/iter_lasso_1500_500_sqrt.txt};
\addplot [thin, mark=otimes, mark size=2.5pt, color=bcyan]
table[x={n}, y={p_1000}] {matlab/iter_lasso_1000_500_sqrt.txt};
\end{axis}
\end{tikzpicture}
}
& 
\subfloat[][$n=1000$]{
\begin{tikzpicture}[scale=1.0]
\begin{axis}[height=5.5cm,width=6.5cm, legend pos=north east,
ylabel style={yshift=-0.75em},
xlabel style={yshift=0.75em,font=\scriptsize},
tick label style={font=\tiny},
legend style={font=\tiny},
legend entries={
$c/\sqrt{t}$,
$p=4000$,
$p=3000$,
$p=2000$,
},
xmin =0, xmax=2000, ymin=1e1, ymax=5e6, ymode=log, mark repeat =
203,
xlabel=Iteration $t$, 
ylabel=$\scriptstyle{-\Delta f(x_t;\alpha_t(v_t-x_t))/\alpha_t}$,
grid
]
\addplot [thick, mark=o, mark size=2.5pt, color=black]
table[x={n}, y={sqrt}] {matlab/iter_lasso_1000_sqrt.txt};
\addplot [thin, mark=square, mark size=2.5pt, color=bblu]
table[x={n}, y={p_4000}] {matlab/iter_lasso_4000_1000_sqrt.txt};
\addplot [thin, mark=triangle, mark size=2.5pt, color=bred]
table[x={n}, y={p_3000}] {matlab/iter_lasso_3000_1000_sqrt.txt};
\addplot [thin, mark=otimes, mark size=2.5pt, color=bcyan]
table[x={n}, y={p_2000}] {matlab/iter_lasso_2000_1000_sqrt.txt};
\end{axis}
\end{tikzpicture}
}
\end{tabular}
\end{center}
\caption{
Convergence behavior of Algorithm~\ref{alg:asfw} with
$\alpha_t=1/\sqrt{1+t}$ 
on the bound-constrained LASSO problem (Section~\ref{sec:lasso})
for various values of $(n,p)$.}
\label{fig:bclasso}
\end{figure}

Furthermore, we also tested the step size $\alpha_t=2/(t+2)$. The optimization history for various values of $n$ and $p$ are shown in
Figure~\ref{fig:bclasso1}. Once more, all optimizations terminated at the optimal solution
with the stopping criterion $\Delta f(x_t,\alpha_t(v_t-x_t))=0$.
The observed convergence rate is $\cO(1/t)$, motivating further
research in this direction.
\begin{figure}[t]
\begin{center}             
\begin{tabular}{c c}
\subfloat[][$n=125$]{
\begin{tikzpicture}[scale=1.0]
\begin{axis}[height=5.5cm,width=6.5cm, legend pos=south west,
ylabel style={yshift=-0.75em}, 
xlabel style={yshift=0.75em,font=\scriptsize},
tick label style={font=\tiny},
legend style={font=\tiny},
legend entries={
$c/t$,
$p=500$,
$p=375$,
$p=250$,
},
xmin =0, xmax=40, ymin=1e-2, ymax=5e6, ymode=log, mark repeat =
8,
xlabel=Iteration $t$, 
ylabel=$\scriptstyle{-\Delta f(x_t;\alpha_t(v_t-x_t))/\alpha_t}$,
grid
]
\addplot [thick, mark=o, mark size=2.5pt, color=black]
table[x={n}, y={t}] {matlab/iter_lasso_125_t.txt};
\addplot [thick, mark=square, mark size=2.5pt, color=bblu]
table[x={n}, y={p_500}] {matlab/iter_lasso_500_125_t.txt};
\addplot [thick, mark=triangle, mark size=2.5pt, color=bred]
table[x={n}, y={p_375}] {matlab/iter_lasso_375_125_t.txt};
\addplot [thick, mark=otimes, mark size=2.5pt, color=bcyan]
table[x={n}, y={p_250}] {matlab/iter_lasso_250_125_t.txt};
\end{axis}
\end{tikzpicture}
}
& 
\subfloat[][$n=250$]{
\begin{tikzpicture}[scale=1.0]
\begin{axis}[height=5.5cm,width=6.5cm, legend pos=south west,
ylabel style={yshift=-0.75em}, 
xlabel style={yshift=0.75em,font=\scriptsize},
tick label style={font=\tiny},
legend style={font=\tiny},
legend entries={
$c/t$,
$p=1000$,
$p=750$,
$p=500$,
},
xmin =0, xmax=50, ymin=1e-2, ymax=5e6, ymode=log, mark repeat =
8,
xlabel=Iteration $t$, 
ylabel=$\scriptstyle{-\Delta f(x_t;\alpha_t(v_t-x_t))/\alpha_t}$,
grid
]
\addplot [thick, mark=o, mark size=2.5pt, color=black]
table[x={n}, y={t}] {matlab/iter_lasso_250_t.txt};
\addplot [thick, mark=square, mark size=2.5pt, color=bblu]
table[x={n}, y={p_1000}] {matlab/iter_lasso_1000_250_t.txt};
\addplot [thick, mark=triangle, mark size=2.5pt, color=bred]
table[x={n}, y={p_750}] {matlab/iter_lasso_750_250_t.txt};
\addplot [thick, mark=otimes, mark size=2.5pt, color=bcyan]
table[x={n}, y={p_500}] {matlab/iter_lasso_500_250_t.txt};
\end{axis}
\end{tikzpicture}
}
\\ 
\subfloat[][$n=500$]{
\begin{tikzpicture}[scale=1.0]
\begin{axis}[height=5.5cm,width=6.5cm, legend pos=south west,
ylabel style={yshift=-0.75em}, 
xlabel style={yshift=0.75em,font=\scriptsize},
tick label style={font=\tiny},
legend style={font=\tiny},
legend entries={
$c/t$,
$p=2000$,
$p=1500$,
$p=1000$,
},
xmin =0, xmax=60, ymin=1e-2, ymax=5e6, ymode=log, mark repeat =
8,
xlabel=Iteration $t$, 
ylabel=$\scriptstyle{-\Delta f(x_t;\alpha_t(v_t-x_t))/\alpha_t}$,
grid
]
\addplot [thick, mark=o, mark size=2.5pt, color=black]
table[x={n}, y={t}] {matlab/iter_lasso_250_t.txt};
\addplot [thick, mark=square, mark size=2.5pt, color=bblu]
table[x={n}, y={p_2000}] {matlab/iter_lasso_2000_500_t.txt};
\addplot [thick, mark=triangle, mark size=2.5pt, color=bred]
table[x={n}, y={p_1500}] {matlab/iter_lasso_1500_500_t.txt};
\addplot [thick, mark=otimes, mark size=2.5pt, color=bcyan]
table[x={n}, y={p_1000}] {matlab/iter_lasso_1000_500_t.txt};
\end{axis}
\end{tikzpicture}
}
&
\subfloat[][$n=1000$]{
\begin{tikzpicture}[scale=1.0]
\begin{axis}[height=5.5cm,width=6.5cm, legend pos=south west,
ylabel style={yshift=-0.75em}, 
xlabel style={yshift=0.75em,font=\scriptsize},
tick label style={font=\tiny},
legend style={font=\tiny},
legend entries={
$c/t$,
$p=4000$,
$p=3000$,
$p=2000$,
},
xmin =0, xmax=90, ymin=1e-2, ymax=5e6, ymode=log, mark repeat =
8,
xlabel=Iteration $t$, 
ylabel=$\scriptstyle{-\Delta f(x_t;\alpha_t(v_t-x_t))/\alpha_t}$,
grid
]
\addplot [thick, mark=o, mark size=2.5pt, color=black]
table[x={n}, y={t}] {matlab/iter_lasso_1000_t.txt};
\addplot [thick, mark=square, mark size=2.5pt, color=bblu]
table[x={n}, y={p_4000}] {matlab/iter_lasso_4000_1000_t.txt};
\addplot [thick, mark=triangle, mark size=2.5pt, color=bred]
table[x={n}, y={p_3000}] {matlab/iter_lasso_3000_1000_t.txt};
\addplot [thick, mark=otimes, mark size=2.5pt, color=bcyan]
table[x={n}, y={p_2000}] {matlab/iter_lasso_2000_1000_t.txt};
\end{axis}
\end{tikzpicture}
}
\end{tabular}
\end{center}
\caption{
Convergence behavior of Algorithm~\ref{alg:asfw} with
$\alpha_t=2/(t+2)$ 
on the
bound-constrained LASSO problem (Section~\ref{sec:lasso}) for
various values of $(n,p)$.}
\label{fig:bclasso1}
\end{figure}
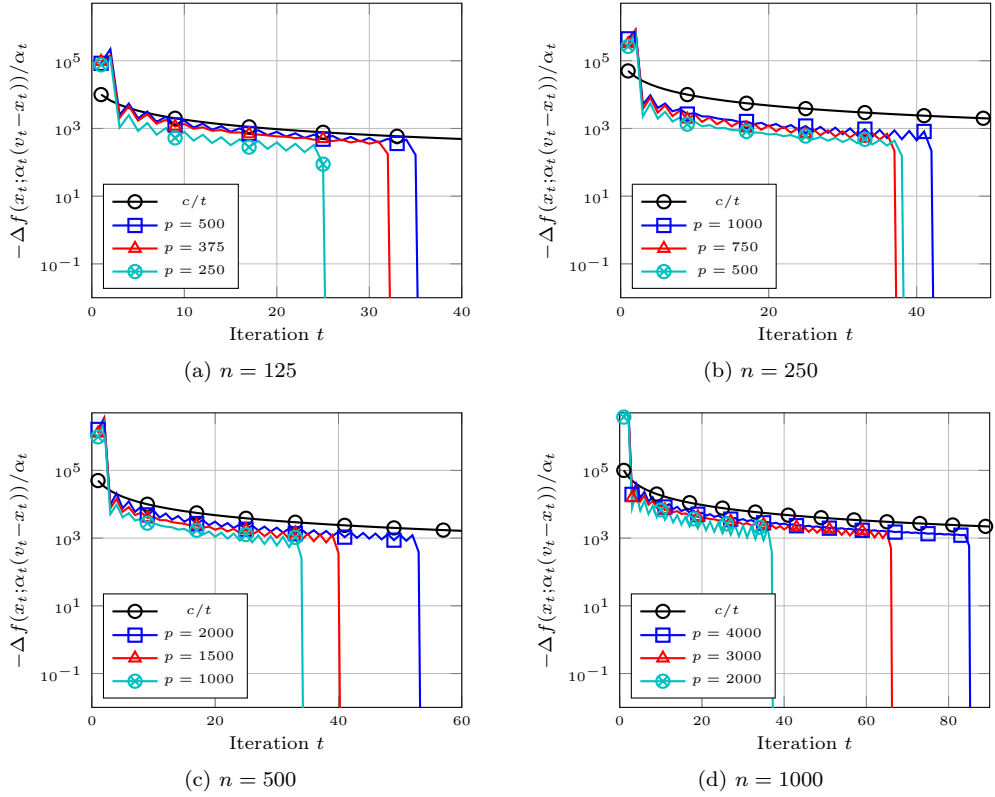

For the second setting, we studied a LASSO problem where the
feasible set is given by
\begin{align*}
\tilde{C} = \{ x \in \R^n \,|\,-5 \le x_{0}\le x_1 \le \cdots \le x_n \le 5\}\;.
\end{align*}
This setting corresponds to the requirement that the parameters
grow monotonically, which is reasonable for example in some climate
models \cite{GaKiZh18}. As initial point, we use 
\begin{align*}
(x_0)_i = -1+\frac{2(i-1)}{n-1}\qquad 1\le i \le n\;,
\end{align*}
i.e., a feasible vector with equality distributed values from $-1$
to $1$.

For the combinations $(n,p)\in\{125\}\times\{250,375,500\}$,
$(n,p)\in\{250\}\times\{500,750,1000\}$, $(n,p)\in\{500\}\times
\{1500,2000\}$ and the step size  $\alpha_t=1/\sqrt{1+t}$, Algorithm~\ref{alg:asfw} determined the optimal
solution $x_* =0\in \R^n$ within two iterations, hence we do not
plot the convergence for this experiment.
Note that all constraints except for the lower and upper bound
are active at the optimal solution found. Using the step size
$\alpha_t=2/(t+2)$, we also observed the same convergence rate
of $\cO(1/t)$ as in Figure~\ref{fig:bclasso1}. Near the end of our
experiments in both Figures~\ref{fig:bclasso} and
\ref{fig:bclasso1}, the algorithm terminates because 
the generalized Frank-Wolfe gap is zero. This sudden stopping is
reminiscent of the smooth Frank-Wolfe literature
\cite{CGFWSurvey2022} -- in certain simple settings (e.g., for a
linear or quadratic objective when the algorithm reaches the
optimal face of $C$), one iteration of the Frank-Wolfe algorithm
 exactly solves the optimization problem. Also note that for $n = 250$  the termination order of the algorithm for the different values of $p$ changes. While the early termination may happen when using
 Frank-Wolfe algorithms, we actually have no guarantee that the point of termination has to satisfy
an order based on the dimension.

\section{Summary and outlook}%
\label{sec:final}

Even nowadays the solution of nonsmooth constrained optimization
problems forms a challenge and correspondingly there is no
off-the-shelf
algorithm available. For a compact convex feasible set, we have shown that
Algorithm~\ref{alg:asfw}, which appears to be the first connection between
the fields of abs-smooth optimization and conditional gradient
algorithms, can exhibit the same per-iteration rate of convergence
as the Frank-Wolfe algorithm for smooth nonconvex objectives.

We have shown that, in generalizing from the smooth setting to the
abs-smooth setting, the Frank-Wolfe gap becomes nonlinear and
nonsmooth, and computing this gap necessitates the solution of a
piecewise linear minimization problem as opposed to linear
minimization in the smooth case. 
Our methodology stands in contrast to recent approaches which
identify subclasses of nonsmooth functions for which, when one
performs linear minimization against a subgradient, the Frank-Wolfe
algorithm converges \cite{AN22,GH16,O23}. The approaches in
\cite{AN22,GH16} only consider convex objectives, and \cite{O23}
considers a class of functions which does not even contain the
convex piecewise linear counterexample by Nesterov
\cite[Example~1]{Ne18}. Instead of finding a smaller class of
functions for which subgradient-approaches work, we have
generalized the Frank-Wolfe algorithm itself and shown that it will
still converge on a broader class of functions.

The proposed algorithm can be implemented easily based on an AD tool
that provides the required abs-linearization and an LP solver.
The numerical illustrations in Sections~\ref{subsec:tests} and
\ref{sec:lasso} exhibit that the theory developed in this work is
consistent with the rates observed in practice. 
Furthermore, we
observed experimentally that improved rates are possible with the
step size strategy $\alpha_t=2/(t+2)$.

As mentioned earlier, the current form of our algorithm comes with
the drawbacks that (A) storing $\Delta f(\overline{x};\cdot)$
scales quadratically with $s$, and (B) computing $\Delta
f(\overline{x};\cdot)$ requires roughly an amount of work proportional
to computing a gradient. While our algorithm is effective on
moderately-sized problems, these are the central bottlenecks
preventing the algorithm from solving very high-dimensional
problems. However, we are hopeful that (A) can be resolved because
we have observed that the matrices decribing $\Delta f(\overline{x};\cdot)$
are often very sparse. Furthermore, we are hopeful that $\Delta
f(\overline{x};\cdot)$ could be reasonably approximated using
block activation strategies similar to SGD or prox-based methods,
which would yield an avenue for the resolution of (B). If these
algorithmic challenges are addressed, our approach would be more
readily applicable on very large-scale problems.

As part of this line of work, we are left with several
questions which we are eager to study in future work. Firstly,
based on numerical experimentation and existing Frank-Wolfe
literature, we believe that showing an $\cO(1/t)$ convergence rate
may be possible with a step size strategy of $\alpha_t=2/(t+2)$.
We also believe that
the convergence analysis in Theorem~\ref{theo:t} could be refined
to directly incorporate the number of switching variables $s$. This
would be particularly useful for improving convergence rates for
functions with a low number of switching variables.
Finally, our numerics indicate that Algorithm~\ref{alg:asfw} works
as long as one has a local minimizer to our
ASFW-subproblem~\eqref{e:pwp}. If this is true in general, our
approach of Algorithm~\ref{alg:asfw} using Algorithm~\ref{alg:aasm}
as a subproblem solver would always converge to a solution, whether
or not the objective function's piecewise linear model is convex.

\section*{Acknowledgments}

The authors thank the Deutsche Forschungsgemeinschaft for their support within Projects A05 and B10 in the Sonderforschungsbereich/Transregio 154 {\em Mathematical Modelling, Simulation and Optimization using the Example of Gas Networks} (project ID: 239904186). 
The work was funded partly  by the Deutsche Forschungsgemeinschaft (DFG, German Research
Foundation) under Germany's Excellence Strategy -- The Berlin Mathematics
Research Center MATH+ (EXC-2046/1, project ID: 390685689).
The data that support the findings of this study are available from the corresponding author upon request.

Conflicts of Interest: The authors declare no conflict of interest.

\printbibliography{}

\end{document}